\documentclass[12pt]{amsart}

\usepackage{amsmath,amsfonts,amsthm,amssymb,mathdots,fullpage,color,enumerate}

\newtheorem{theorem}{Theorem}[section]
\newtheorem{lemma}[theorem]{Lemma}
\newtheorem{prop}[theorem]{Proposition}

\theoremstyle{definition}
\newtheorem{definition}[theorem]{Definition}

\theoremstyle{remark}
\newtheorem{remark}{Remark}

\newcommand{\C}{\mathbb C}

\newcommand{\Z}{\mathbb Z}
\newcommand{\Q}{\mathbb Q}

\newcommand{\A}{\mathbb A}

\newcommand{\p}{\mathfrak p}
\newcommand{\g}{\mathfrak g}

\newcommand{\cO}{\mathcal O}

\newcommand{\cL}{\mathcal L}
\newcommand{\cF}{\mathcal F}

\newcommand{\cP}{\mathcal P}

\newcommand{\cH}{\mathcal H}

\newcommand{\Ad}{\text{Ad}}
\newcommand{\End}{\text{End}}

\newcommand{\tl}{\widetilde{\lambda}}
\newcommand{\tc}{\widetilde{\chi}}

\newcommand{\be}{\begin{equation}}
\newcommand{\ee}{\end{equation}}
\newcommand{\bes}{\begin{equation*}}
\newcommand{\ees}{\end{equation*}}

\newcommand{\ba}{\begin{eqnarray}}
\newcommand{\ea}{\end{eqnarray}}
\newcommand{\bas}{\begin{eqnarray*}}
\newcommand{\eas}{\end{eqnarray*}}

\title{Subconvexity for $L$-functions on ${\rm U}(n) \times {\rm U}(n+1)$ in the depth aspect}
\author{Simon Marshall}
\address{Department of Mathematics\\
University of Wisconsin -- Madison\\
480 Lincoln Drive\\
Madison\\
WI 53706, USA}
\email{marshall@math.wisc.edu}
\thanks{The author was supported by NSF grants DMS-1902173 and DMS-1954479, and was also supported as the Neil Chriss and Natasha Herron Chriss Founders' Circle Member while at the IAS in 2017--18.}

\begin{document}

\maketitle

\begin{abstract}

Let $E/F$ be a CM extension of number fields, and let $H < G$ be a unitary Gan--Gross--Prasad pair defined with respect to $E/F$ that is compact at infinity.  We consider a family $\cF$ of automorphic representations of $G \times H$ that is varying at a finite place $w$ that splits in $E/F$.  We assume that the representations in $\cF$ satisfy certain conditions, including being tempered and distinguished by the GGP period.  For a representation $\pi \times \pi_H \in \cF$ with base change $\Pi \times \Pi_H$ to ${\rm GL}_{n+1}(E) \times {\rm GL}_n(E)$, we prove a subconvex bound
\[
L(1/2, \Pi \times \Pi_H^\vee) \ll C(\Pi \times \Pi_H^\vee)^{1/4 - \delta}
\]
for any $\delta < \tfrac{1}{4n(n+1)(2n^2 + 3n + 3)}$.  Our proof uses the unitary Ichino--Ikeda period formula to relate the central $L$-value to an automorphic period, before bounding that period using the amplification method of Iwaniec--Sarnak.

\end{abstract}

\section{Introduction}

The purpose of this paper is to prove a subconvex bound for certain families of $L$-functions on ${\rm U}(n) \times {\rm U}(n+1)$, by first using the unitary Ichino--Ikeda period formula to relate the central $L$-value to an automorphic period, and then using the amplification method of Iwaniec--Sarnak to bound that period.    We begin by stating the period bound we shall prove, as it holds for families of forms that are, in principle, more general than those for which our subconvex bound holds.

\subsection{The period bound}

We shall consider the following unitary GGP periods.  Let $E/F$ be a CM extension of number fields.  Let $V$ be a non-degenerate Hermitian space of dimension $n+1$ with respect to $E/F$ that is positive definite at all infinite places, and let $V_H \subset V$ be an $n$-dimensional subspace, which is automatically non-degenerate.  We define the unitary groups $G = {\rm U}(V)$ and $H = {\rm U}(V_H)$.  We shall consider periods of automorphic forms on $G \times H$ over the diagonally embedded copy of $H$.  If $f$ and $f_H$ are continuous functions on the adelic quotients $[G]$ and $[H]$ we introduce the notation
\[
\cP(f, f_H) = \int_{ [H] } f(h) f_H(h) dh
\]
for this period.

We next define the family of automorphic representations $\pi \times \pi_H$ on $G \times H$ we shall work with.  Let $S$ be the finite set of places chosen in Section \ref{sec:notation1}, which contains the infinite places, and all places at which $G$ and $H$ are ramified.  We also let $K_v$ and $K_{H,v}$ be the compact open subgroups of $G_v$ and $H_v$ chosen there, which are assumed to be hyperspecial for $v \notin S$.  Let $w \notin S$ be a place at which $E/F$ splits.  This is the place at which our representations will vary, and will be distinguished throughout the paper.  At $w$, we will assume that $\pi_w$ and $\pi_{H,w}$ are principal series representations induced from a \textit{generic pair} of characters.  To define this, let $\cO_w$ be the completion of the ring of integers of $F$ at $w$, $\p$ the maximal ideal of $\cO_w$, and $q$ the order of the residue field.

\begin{definition}
\label{genericpair}

Let $l \geqslant 1$ be given.  Let $\chi$ and $\chi_H$ be complex characters of $(F_w^\times)^{n+1}$ and $(F_w^\times)^n$ respectively, which are not assumed to be unitary.

\begin{enumerate}

\item We say that $\chi = (\chi_1, \ldots, \chi_{n+1})$ is generic of conductor $\p^l$ if it is trivial on $(1 + \p^l)^{n+1}$, and the components $\chi_i$ have the property that $\chi_i \chi_j^{-1}$ has conductor $\p^l$ for all $i \neq j$.  We make the analogous definition for $\chi_H$.

\item We say that $(\chi, \chi_H)$ is a generic pair of conductor $\p^l$ if both $\chi$ and $\chi_H$ are generic of conductor $\p^l$, and if $\chi_i \chi_{H,j}^{-1}$ has conductor $\p^l$ for all $i$ and $j$

\end{enumerate}

\end{definition}

Given a generic pair $(\chi, \chi_H)$ of conductor $\p^l$, we may form the representation $\text{Ind}_{B_w}^{G_w} \chi$ of $G_w$ unitarily induced from a Borel subgroup $B_w$, which we denote by $I(\chi)$, and define the representation $I(\chi_H)$ of $H_w$ similarly.  Note that $I(\chi)$ and $I(\chi_H)$ are irreducible, by a theorem of Bernstein--Zelevinsky \cite[Theorem 4.2]{BZ}.  These will be our choices for $\pi_w$ and $\pi_{H,w}$.

 In order to choose our test vectors at $w$, we shall also need the notion of a compatible pair of microlocal lift vectors $\phi \in I(\chi)$ and $\phi_H \in I(\chi_H)$.  We define this formally in Definition \ref{compatdef}, but summarize that definition here.  If $l$ is even, then in Definition \ref{chitypedef} we define the notion of a $\chi$-type, which is a pair $(J, \widetilde{\lambda})$ where $J$ is a compact open subgroup of $K_w$ and $\widetilde{\lambda}$ is a character of $J$.  By a theorem of Roche \cite[Thm. 7.7]{Ro}, they are in fact $\mathfrak{s}$-types, in the sense of Bushnell--Kutzko, for the supercuspidal data $\mathfrak{s}$ naturally associated with $\chi$, but we will not use this fact.  We then say that $\phi \in I(\chi)$ is a \textit{microlocal lift vector} associated to $(J, \widetilde{\lambda})$ if $\phi$ transforms by $\widetilde{\lambda}$ under $J$.  (We will show that a nonzero microlocal lift vector is associated to at most one $\chi$-type.)

We define the notion of a $\chi_H$-type $(J_H, \tl_H)$, and microlocal lift vectors $\phi_H \in I(\chi_H)$, for $H_w$ in a similar way.  We say that $(J, \widetilde{\lambda})$ and $(J_H, \tl_H)$ are compatible $(\chi, \chi_H)$-types if $\tl$ and $\tl_H$ agree on $J \cap J_H$ (which will be equal to the principal congruence subgroup $K_{H,w}(\p^l)$ by Lemma \ref{Atrans}), and that microlocal lift vectors $\phi$ and $\phi_H$ are compatible if their types are.

We may now define the family $\cF_\text{P}$ of automorphic representations for which our period bound holds.

\begin{definition}

Let $\cF_\text{P}$ be the set of automorphic representations $\pi \times \pi_H$ of $G \times H$ satisfying the following conditions:

\begin{enumerate}

\item[(C1)] $\pi_v$ and $\pi_{H,v}$ have nonzero fixed vectors under $K_v$ and $K_{H,v}$ for finite places $v \neq w$.  In particular, $\pi_v$ and $\pi_{H,v}$ are unramified outside of $S \cup \{ w \}$.

\item[(C2)] $\pi_w$ and $\pi_{H,w}$ are principal series representations induced from a generic pair of characters of square conductor $\p^{2l}$ for some $l \geqslant 1$.

\end{enumerate}

\end{definition}

Our main automorphic period bound is the following.

\begin{theorem}
\label{mainperiod}

Let $\pi \times \pi_H \in \cF_\textup{P}$.  We denote the representation $\pi_\infty$ of $G_\infty$ by $\mu$, and let $d_\mu$ denote its dimension.  Let $\phi = \otimes_v \phi_v \in \pi$ and $\phi_H = \otimes_v \phi_{H,v} \in \pi_H$ be $L^2$-normalized vectors satisfying the following:

\begin{enumerate}

\item $\phi_v$ and $\phi_{H,v}$ are invariant under $K_v$ and $K_{H,v}$ for finite places $v \neq w$.

\item $\phi_w$ and $\phi_{H,w}$ are compatible microlocal lift vectors in the sense of Definition \ref{compatdef}.

\end{enumerate}
Then
\[
\cP_H(\phi, \overline{\phi_H} ) \ll d_\mu q^{(n/2 - \delta)l}
\]
for any $\delta < \tfrac{1}{4n^2 + 6n + 6}$.

\end{theorem}

Note that the dependence of the constant on $\mu$ in Theorem \ref{mainperiod}, and Proposition \ref{trivialbound} below, is not sharp.

\begin{remark}

Theorem \ref{mainperiod} makes use of the fact that $\pi$ and $\pi_H$ are tempered at all places of $F$ that split in $E$, which follows from work of Caraiani \cite{Caraiani} and Labesse \cite{Lab} as we establish in Section \ref{sec:tempered}.  However, this is not essential -- one may obtain a version of Theorem \ref{mainperiod} with a weaker exponent if one only assumes that $\pi$ and $\pi_H$ are $\theta$-tempered at almost all split places, for some $0 \leqslant \theta < 1/2$.  (We refer to Section \ref{sec:spherical} for the definition of $\theta$-temperedness.)  The modifications required in this case are simple: when proving the bound of Proposition \ref{diagonal bound} for the diagonal contribution on the geometric side of the relative trace formula, one applies Lemma \ref{Hecke restriction} with this $\theta$, and obtains a bound of $N^{1 + 2\theta + \epsilon} d_\mu^2 q^{nl}$ instead of $N^{1 + \epsilon} d_\mu^2 q^{nl}$.  The rest of the proof proceeds as before, and one obtains an exponent of $\delta < \tfrac{1-2\theta}{ 2(2n^2 +3n +3 - 4\theta) }$ in the final period bound.

We point this out because $\theta$-temperedness, for some $0 \leqslant \theta < 1/2$, is known for a much larger collection of automorphic forms than temperedness is.  For instance, $\theta$ temperedness is known for cusp forms on ${\rm GL}_N$, by the bounds of \cite{LRS} towards the generalized Ramamujan conjecture, and can be deduced for certain forms on classical groups by endoscopy or base change arguments.

This is in contrast with Theorem \ref{mainsubconvex}, in which temperedness is needed to apply the Ichino--Ikeda formula of \cite{BPCZ}.

\end{remark}

The bound $\cP_H(\phi, \overline{\phi_H} ) \ll q^{nl/2}$ may be thought of as the trivial bound for this period, and it corresponds to the convex bound for $L(1/2, \Pi \times \Pi_H^\vee)$ in Theorem \ref{mainsubconvex} below.  It holds in greater generality than Theorem \ref{mainperiod}, and only uses `local' information about $\phi$ and $\phi_H$, by which we mean information about how they transform under a compact subgroup.  As this may be of independent interest, we state it here as a separate result.

\begin{prop}
\label{trivialbound}

Let $(\chi, \chi_H)$ be a generic pair of characters of conductor $\p^{2l}$.  Let $\phi \in L^2( [G])$ and $\phi_H \in L^2([H])$ be $L^2$-normalized functions satisfying the following:

\begin{enumerate}

\item $\phi$ is $\mu$-isotypic under the action of $G_\infty$, for some irreducible representation $\mu$ of $G_\infty$.

\item $\phi$ and $\phi_{H}$ are invariant under $K_v$ and $K_{H,v}$ for finite places $v \neq w$.

\item There are compatible $(\chi, \chi_H)$-types $(J, \tl)$ and $(J_H, \tl_H)$, in the sense of Definition \ref{compatdef}, such that $\phi$ transforms under $J$ according to $\tl$, and likewise for $\phi_H$.

\end{enumerate}
Then $\cP_H(\phi, \overline{\phi_H} ) \ll d_\mu q^{nl/2}$, where $d_\mu$ denotes the dimension of $\mu$.

\end{prop}

\subsection{The subconvex bound}

Our subconvex bound will hold for the following family of representations.

\begin{definition}

Let $\cF_\text{SC}$ be the set of automorphic representations $\pi \times \pi_H$ of $G \times H$ satisfying (C1) and (C2), as well as the following extra conditions:

\begin{enumerate}

\item[(C3)] $\pi_\infty \in \Theta_\infty$, and $\pi_{H,\infty} \in \Theta_{H,\infty}$, where $\Theta_\infty$ and $\Theta_{H,\infty}$ are finite sets of irreducible representations of $G_\infty$ and $H_\infty$.

\item[(C4)] $\pi$ and $\pi_H$ are tempered, and admit weak base change lifts $\Pi$ and $\Pi_H$ to ${\rm GL}_{n+1}$ and ${\rm GL}_n$ that satisfy the conditions for being a Hermitian Arthur parameter given in \cite[Section 1.1.3]{BPCZ}.

\item[(C5)] $(\pi_v, \pi_{H,v})$ is locally distinguished for all $v$, in the sense that there is a nonzero $H_v$-intertwining map $\pi_v \to \pi_{H,v}$.

\end{enumerate}

\end{definition}

\begin{remark}
\label{tempered2}

We note that condition (C4) is implied by the work of Mok \cite{Mo} and Kaletha--Minguez--Shin--White \cite{KMSW}, although this is currently conditional.  Moreover, the unconditional results of Caraiani and Labesse do not suffice here; while \cite{Caraiani} implies that $\Pi$ and $\Pi_H$ are tempered everywhere, \cite{Lab} does not establish local-global compatibility at all places, so we cannot deduce that $\pi$ and $\pi_H$ are tempered everywhere.  Likewise, while \cite{Lab} proves the existence of the weak base change, it does not give the extra conditions of \cite[Section 1.1.3]{BPCZ}, which Beuzart-Plessis informs us are needed for their proof.

\end{remark}

Our main subconvex bound is the following.

\begin{theorem}
\label{mainsubconvex}

Let $\pi \times \pi_H \in \cF_\textup{SC}$, and let $\Pi \times \Pi_H$ be its base change to ${\rm GL}_{n+1}(\A_E) \times {\rm GL}_n(\A_E)$.  Then the $L$-function $L(1/2, \Pi \times \Pi_H^\vee)$ satisfies the bound
\[
L(1/2, \Pi \times \Pi_H^\vee) \ll q^{l n (n+1) (1 - 4 \delta )} \ll C(\Pi \times \Pi_H^\vee)^{1/4 - \delta}
\]
for any $\delta < \tfrac{1}{4n(n+1)(2n^2 + 3n + 3)}$.

\end{theorem}

\begin{remark}
\label{conductor}

The contribution to the conductor of $\Pi \times \Pi_H$ from the places above $w$ is equal to $q^{4l n (n+1)}$.  Because $\pi \times \pi_H$ has bounded ramification away from $w$, the same should be true of $\Pi \times \Pi_H$, which would imply that the conductor away from $w$ is bounded and hence that $C(\Pi \times \Pi_H^\vee) \sim q^{4l n (n+1)}$.  However, to the best of our knowledge, actually proving this ramification bound requires the (conditional) results of \cite{KMSW}.  Fortunately, Theorem \ref{mainsubconvex} only requires the weaker statement $C(\Pi \times \Pi_H^\vee) \gg q^{4l n (n+1)}$.

\end{remark}

\begin{remark}

In deducing Theorem \ref{mainsubconvex} from Theorem \ref{mainperiod}, we have used the result \cite[Lemma 5.2]{Ne1} of Nelson to provide a lower bound for the local periods appearing in the Ichino--Ikeda formula at ramified finite places.  This result lets us handle the most general set of local factors of $\pi$ and $\pi_H$ at these places, and could be avoided by assuming that e.g. $E/F$ and $\pi \times \pi_H$ are unramified at finite places away from $w$.

\end{remark}

\subsection{Relation to other results}

The literature on the subconvexity problem is extensive, and we refer to \cite{IS, Mu4} for a survey of known results.  On ${\rm GL}_3$ and ${\rm GL}_3 \times {\rm GL}_2$, we mention the bounds of Blomer--Buttcane \cite{BB1}, Li \cite{Li}, and Munshi \cite{Mu1, Mu2, Mu3}.  In general rank, we mention the results of Hu--Nelson \cite{HN}, Nelson \cite{Ne1, Ne2}, and Liyang Yang \cite{Liyang}, which use a relative trace formula approach related to that of this paper.

We proved Theorem \ref{mainperiod} in 2018, and gave a detailed lecture at the IAS \cite{lecture} on the methods used.  This paper, which has been significantly delayed, implements those methods.  

\subsection{Acknowledgements}

We would like to thank Rapha\"{e}l Beuzart-Plessis, Farrell Brumley, Paul Nelson, Peter Sarnak, and Sug Woo Shin for helpful conversations.  We would also like to thank the IAS for providing an excellent working environment during the 2017--18 special year on Locally Symmetric Spaces: Analytical and Topological Aspects.

\section{Microlocal lift vectors}

In this section, we prove the local representation-theoretic results we shall need at the place $w$.  If $\chi$ is a generic character, we shall define the notion of a microlocal lift vector in the representation $I(\chi)$ and establish their basic properties such as existence, uniqueness, and the support of their diagonal matrix coefficient.  Then, for a generic pair $(\chi, \chi_H)$, we define the notion of compatible microlocal lift vectors in $I(\chi)$ and $I(\chi_H)$, prove that such compatible vectors exist, and establish some transversality results for them.

\subsection{Notation}

As we shall work locally at $w$ in this section, we drop the place from the notation.  We let $F$ be a $p$-adic field, with integers $\cO$ and maximal ideal $\p$.  We let $q$ be the order of the residue field, 
and $\varpi$ a uniformizer.  In the first part of this section, we shall work on the group ${\rm GL}_n(F)$.  We let $B$, $T$, and $K$ be the standard Borel, diagonal subgroup, and maximal compact in ${\rm GL}_n(F)$.  For $l \geqslant 1$ we let $K(l)$ be the standard principal congruence subgroup of $K$.  Let $\g \simeq M_n(F)$ be the Lie algebra of ${\rm GL}_n(F)$.

If $g \in G$, we denote the conjugation action of $g$ on elements $j \in G$ and subsets $J \subset G$ by $j^g := g j g^{-1}$ and $J^g := g J g^{-1}$.  If $J$ is a subgroup of $G$ and $\lambda$ is a character of $J$, then $\lambda^g$ is the character of $J^g$ given by $\lambda^g(j) = \lambda( g^{-1} j g)$.

\subsection{A review of characters of compact subgroups}
\label{sec:characters}

We shall use a correspondence between characters of certain compact open subgroups of ${\rm GL}_n(F)$, and cosets in $\g$, which we recall from e.g. \cite[Section 2]{Ho}.  Let $\cL = M_n(\cO) \subset \g$.  We have $K(l) = 1 + \p^l \cL$, and this induces a group isomorphism $K(l) / K(2l) \simeq \p^l \cL / \p^{2l} \cL$.

Let $\theta$ be an additive character of $F$ of conductor $\cO$, and let $\langle \, , \rangle$ be the trace pairing on $\g$.  We define an isomorphism $\tau : \g \to \widehat{\g}$, where $\widehat{\g}$ is the Pontryagin dual, by $\tau[X](Y) = \theta( \langle X, Y \rangle)$.  The dual of $\p^l \cL$ under $\tau$ is $\p^{-l} \cL$, and so $\tau$ induces an isomorphism
\[
\p^{-2l} \cL / \p^{-l} \cL \xrightarrow{\sim} \widehat{\p^l \cL / \p^{2l} \cL} \simeq \widehat{K(l) / K(2l)}.
\]
It will be convenient for us to rescale this to an isomorphism
\[
\log^*: \cL / \p^l \cL \xrightarrow{\sim} \widehat{K(l) / K(2l)}
\]
defined by $\log^*[X](1+Y) = \theta( \varpi^{-2l} \langle X, Y \rangle)$, with inverse
\[
\exp^*: \widehat{K(l) / K(2l)} \xrightarrow{\sim} \cL / \p^l \cL \simeq M_n(\cO / \p^l).
\]
(The notation is intended to evoke the pullback of characters via the exponential map.) Specializing to the case $n = 1$, we obtain an isomorphism $\exp^*_1$ from the dual group of $(1 + \p^l)^\times / (1 + \p^{2l})^\times$ to $\cO / \p^l$.  We may use this to reformulate the definition of a generic character of square conductor.  If $\chi = (\chi_1, \ldots, \chi_n)$ is a character of $(F^\times)^n$ with all $\chi_i$ trivial on $1 + \p^{2l}$, we may define
\[
D(\chi) = \text{diag}( \exp^*_1(\chi_1), \ldots, \exp^*_1(\chi_1)) \in M_n(\cO / \p^l),
\]
where we have implicitly restricted each $\chi_i$ to $1 + \p^l$.  Then $\chi$ is generic of conductor $\p^{2l}$ if and only if $D(\chi)$ is regular when reduced modulo $\p$. Likewise, if $\chi$ and $\chi_H$ are characters of $(F^\times)^{n+1}$ and $(F^\times)^n$ all of whose coordinates are trivial on $1 + \p^{2l}$, then $(\chi, \chi_H)$ is a generic pair of conductor $\p^{2l}$ if and only if the mod $\p$ reductions of $D(\chi)$ and $D(\chi_H)$ are regular with distinct eigenvalues.

\subsection{Microlocal lift vectors}

Let $\chi$ be a generic character of $(F^\times)^n$ of square conductor $\p^{2l}$, and let $I(\chi) = \text{Ind}_{B}^{{\rm GL}_n} \chi$ be the associated principal series representation of ${\rm GL}_n(F)$.  We shall define microlocal lift vectors in $I(\chi)$.  These vectors will be associated with certain pairs $(J, \widetilde{\lambda})$, where $J$ is a compact open subgroup of $K$ and $\widetilde{\lambda}$ is a character of $J$, and will be characterized by the property that they transform by $\widetilde{\lambda}$ under $J$.  The pairs $(J, \widetilde{\lambda})$ are called $\chi$-types, and are defined in Definition \ref{chitypedef} below.  They are in fact $\mathfrak{s}$-types for the supercuspidal data $\mathfrak{s} = (T, \chi)$ in the sense of Bushnell--Kutzko, by a theorem of Roche \cite[Thm. 7.7]{Ro}, but we will not use this fact.

To define $\chi$-types, we begin by constructing the so-called standard $\chi$-type.  Let $T_c$ be the maximal compact subgroup of $T$, and define $\widetilde{T}(l) = T_c K(l)$ to be the depth-$l$ fattening of $T_c$.

\begin{lemma}

There exists a unique character $\widetilde{\chi}$ of $\widetilde{T}(l)$ that agrees with $\chi$ on $T_c$.  Moreover, $\widetilde{\chi}$ is trivial on $K(2l)$.

\end{lemma}

\begin{proof}

Any element of $\widetilde{T}(l)$ has the form $X + Y$ where $X$ is diagonal with entries in $\cO^\times$, and $Y \in \p^l M_n(\cO)$ has trivial diagonal entries.  With this notation, we define $\widetilde{\chi}(X + Y) = \chi(X)$.  To show $\widetilde{\chi}$ is a character, we have
\[
\widetilde{\chi}( (X_1 + Y_1)(X_2 + Y_2)) = \widetilde{\chi}( X_1X_2 + X_1 Y_2 + Y_1 X_2 + Y_1 Y_2),
\]
and because $X_1 Y_2 + X_2 Y_1$ has trivial diagonal entries and $Y_1 Y_2 \in \p^{2l}M_n(\cO)$, this is equal to $\chi(X_1 X_2) = \widetilde{\chi}(X_1 + Y_1) \widetilde{\chi}(X_2 + Y_2)$ as required.  It is clear that $\widetilde{\chi}$ is trivial on $K(2l)$.

To show uniqueness, suppose there are two such characters $\widetilde{\chi}_1$ and $\widetilde{\chi}_2$.  Their quotient $\widetilde{\chi}_1 / \widetilde{\chi}_2$ is trivial on $T_c$, and on the commutator subgroup $[\widetilde{T}(l), \widetilde{T}(l)]$, so it is enough to show that these two groups generate $\widetilde{T}(l)$.  If $E_{ij}$ is the $(i,j)$th elementary matrix, then $[\widetilde{T}(l), \widetilde{T}(l)]$ contains the subgroup $I + \p^l E_{ij}$, as this is the set of commutators of $I + E_{ij}$ with $T_c$.  It may be seen that $T_c$ and the subgroups $I + \p^l E_{ij}$ generate $\widetilde{T}(l)$, which completes the proof.

\end{proof}

We next define $\chi$-types to be conjugates of $( \widetilde{T}(l), \widetilde{\chi})$ under $K$.

\begin{definition}
\label{chitypedef}

Let $(J, \widetilde{\lambda})$ be a pair consisting of an open subgroup $J < K$ and a character of $J$.  We say $(J, \widetilde{\lambda})$ is a $\chi$-type if there is $k \in K$ such that $J = \widetilde{T}(l)^k$ and $\widetilde{\lambda} = \widetilde{\chi}^k$.  We call $( \widetilde{T}(l), \widetilde{\chi})$ the standard $\chi$-type.

\end{definition}

The restriction of $\widetilde{\chi}$ to $K(l)$ satisfies $\exp^*( \widetilde{\chi}|_{K(l)} ) = D(\chi)$, and hence for any $\chi$-type $(J, \widetilde{\lambda}) = ( \widetilde{T}(l)^k, \widetilde{\chi}^k )$ we have $\exp^*( \widetilde{\lambda}|_{K(l)} ) = D(\chi)^k$.  We may now state the definition of microlocal lift vectors.

\begin{definition}[Microlocal lifts]
\label{liftdef}

We say that $v \in I(\chi)$ is a microlocal lift vector associated to a $\chi$-type $(J, \widetilde{\lambda})$ if $v$ transforms by $\widetilde{\lambda}$ under $J$.

\end{definition}

We simply say that $v$ is a microlocal lift vector if it is a microlocal lift vector associated to some choice of $\chi$-type.  The first result we shall prove about such vectors is that there is a unique microlocal vector up to scaling associated with each $\chi$-type.  Uniqueness requires our assumption that $\chi$ is generic.

\begin{lemma}
\label{microlocalunique}

For each $\chi$-type $(J, \tl)$, the space of microlocal lift vectors in $I(\chi)$ associated with $(J, \tl)$ is one-dimensional.

\end{lemma}

\begin{proof}

We may assume that $(J, \tl)$ is standard.  We work in the compact model of $I(\chi)$, as the space of functions on $K$ satisfying
\be
\label{Btransf}
f(bk) = \chi(b) f(k) \quad \text{for} \quad b \in B_K := B \cap K.
\ee
We wish to show that there is a unique such function up to scalar that satisfies
\be
\label{atransf}
f(k a) = \widetilde{\chi}(a) f(k) \quad \text{for} \quad a \in \widetilde{T}(l).
\ee
For existence, we may define a function $f$ supported on $B_K \widetilde{T}(l)$ by $f(ba) = \chi(b) \widetilde{\chi}(a)$.  This is well defined, because it is clear that the characters $\chi$ of $B_K$ and $\widetilde{\chi}$ of $\widetilde{T}(l)$ agree on $B_K \cap \widetilde{T}(l) = B \cap \widetilde{T}(l)$.

For uniqueness, we wish to show that if $f$ satisfies (\ref{Btransf}) and (\ref{atransf}), and $f(k) \neq 0$, then $k \in B_K \widetilde{T}(l)$.  Property (\ref{atransf}) implies that $f(k_1 k) = \widetilde{\chi}^k(k_1) f(k)$ for $k_1 \in K(l)$, and comparing this with property (\ref{Btransf}) and our assumption $f(k) \neq 0$ we see that $\widetilde{\chi}^k = \chi$ on $B_K \cap K(l)$.  If we let $X = D(\chi)$, then we have $\exp^*( \widetilde{\chi}) = X$ and $\exp^*(\widetilde{\chi}^k) = X^k$ as above.  The condition that $\widetilde{\chi}^k = \chi$ on $B_K \cap K(l)$ implies that $X^k$ is upper triangular, and equal to $X$ on the diagonal.

Let $\overline{k} \in {\rm GL}_n(\cO/\p^l)$ be the mod $\p^l$ reduction of $k$.  Because $X$ is regular there is some upper triangular matrix $b \in {\rm GL}_n(\cO / \p^l)$ such that $b X b^{-1} = k X k^{-1} \in M_n(\cO / \p^l)$.  It follows that $\overline{k}^{-1} b$ centralizes $X$, and because $X$ is regular this means $\overline{k}^{-1} b$ is diagonal.  Therefore $\overline{k}$ is upper triangular, so that $k \in B_K \widetilde{T}(l)$ as required.

\end{proof}

\begin{lemma}
\label{uniquetype}

A nonzero microlocal lift vector $v$ is associated with at most one $\chi$-type.

\end{lemma}

\begin{proof}

Suppose there are two $\chi$-types $(J_1, \tl_1) = ( \widetilde{T}(l)^{k_1}, \tc^{k_1})$ and $(J_2, \tl_2) = ( \widetilde{T}(l)^{k_2}, \tc^{k_2})$ associated with $v$, for $k_1, k_2 \in K$.  It follows that $\tl_1$ and $\tl_2$ have the same restriction to $K(l)$, which means that $D(\chi)^{k_1} = \exp^*( \tl_1 |_{K(l)} ) = \exp^*( \tl_2 |_{K(l)} ) = D(\chi)^{k_2}$.  This implies that $k_1 k_2^{-1}$ commutes with $D(\chi)$, and hence that $k_1 k_2^{-1} \in \widetilde{T}(l)$ because $D(\chi)$ is regular.  It follows that $(J_1, \tl_1) = (J_2, \tl_2)$ as required.

\end{proof}

We may prove the following lemma in the same way.

\begin{lemma}
\label{typestabilizer}

The stabilizer of a $\chi$-type $(J, \tl)$ is equal to $J$.

\end{lemma}

The final result we shall prove about individual microlocal lift vectors is a bound on the support of their matrix coefficients.  This will be used later to compute local integrals in the Ichino--Ikeda formula.  The following lemma is a sharper version of Lemma 3.3 of \cite{Ho}.

\begin{lemma}
\label{coeffsupport}

Let $v \in I(\chi)$ be a microlocal lift vector associated to the $\chi$-type $( \widetilde{T}(l)^k, \tl^k)$, and let $v^\vee \in I(\chi^{-1}) \simeq I(\chi)^\vee$ be a microlocal lift vector associated to the $\chi$-type $( \widetilde{T}(l)^k, (\tl^k)^{-1})$.  The matrix coefficient $\rho(g) = \langle g v, v^\vee \rangle$ is supported on $K(l) T^k K(l)$.

\end{lemma}

Note that that the set $K(l) T^k K(l)$ is independent of the choice of $k$ by Lemma \ref{typestabilizer}.

\begin{proof}

It suffices to consider the standard $\chi$-type $( \widetilde{T}(l), \tc)$.  The transformation of $v$ and $v^\vee$ under $K(l)$ implies that for $k \in K(l)$ we have $\rho(kg) = \widetilde{\chi}(k) \rho(g)$ and $\rho(gk) = \widetilde{\chi}(k) \rho(g)$.  If we consider the character $\widetilde{\chi}^g$ of $K(l)^g$, we may write the second transformation property as
\[
\rho(kg) = \widetilde{\chi}^g(k) \rho(g) \quad \text{for} \quad k \in K(l)^g.
\]
If $\rho(g) \neq 0$, comparing this with $\rho(kg) = \widetilde{\chi}(k) \rho(g)$ for $k \in K(l)$ implies that $\widetilde{\chi}$ and $\widetilde{\chi}^g$ must agree on $K(l) \cap K(l)^g$.  We will show that this can only happen when $g \in K(l) T K(l)$.

Let $X \in \cL$ be any lift of $\exp^*(\widetilde{\chi}) = D(\chi)$ to characteristic zero.  We now suppose that $\widetilde{\chi} = \widetilde{\chi}^g$ on $K(l) \cap K(l)^g$.  We first show that this implies $X - X^g \in \p^l \cL + \p^l \cL^g$.  To show this, recall that $\widetilde{\chi}$ is given on $K(l) = 1 + \p^l \cL$ by $\widetilde{\chi}(1 + A) = \theta( \varpi^{-2l} \langle X, A \rangle)$, and likewise $\widetilde{\chi}^g$ is given on $K(l)^g = 1 + \p^l \cL^g$ by $\widetilde{\chi}(1 + A) = \theta( \varpi^{-2l} \langle X^g, A \rangle)$.  We have $K(l) \cap K(l)^g = 1 + \p^l( \cL \cap \cL^g)$, and so for $\widetilde{\chi}$ and $\widetilde{\chi}^g$ to agree we must have $\langle A, X - X^g \rangle \in \p^{2l}$ for $A \in \p^l( \cL \cap \cL^g)$, or equivalently $\langle \cL \cap \cL^g, X - X^g \rangle \in \p^l$.  As the dual lattice to $\cL \cap \cL^g$ is $\cL + \cL^g$, this implies $X - X^g \in \p^l\cL + \p^l\cL^g$ as required.

We may reformulate $X - X^g \in \p^l\cL + \p^l\cL^g$ as saying that there are $x_1, x_2 \in \p^l \cL$ such that $X + x_1 = (X + x_2)^g$.  By iterating Lemma \ref{diagadjust}, we see that $X + x_1$ is equal to $Y_1^{k_1}$ for some $k_1 \in K(l)$ and a regular diagonal matrix $Y_1 \in \cL$ that satisfies $Y_1 \equiv X \; (\p^l)$.  Likewise, we may write $X + x_2 = Y_2^{k_2}$. Combining these, we have $Y_1^{k_1} = Y_2^{g k_2}$, which implies that $Y_1 = Y_2$ and $k_1^{-1} g k_2 \in A$.  This gives $g \in K(l) T K(l)$ as required.

\end{proof}

\begin{lemma}
\label{diagadjust}

Let $Y \in \cL$ be diagonal and regular, and let $y \in \p^l \cL$ for some $l \geqslant 1$.  There is $k \in K(l)$ such that $(Y + y)^k = Z + z$, where $Z$ is diagonal and regular, $Y \equiv Z \; (\p^l)$, and $z \in \p^{l+1} \cL$.

\end{lemma}

\begin{proof}

Let $A \in \cL$, so that $1 + \varpi^l A \in K(l)$.  Because $K(l) / K(2l)$ is abelian, we have $(1 + \varpi^l A)^{-1} \equiv 1 - \varpi^l A \; ( \p^{l+1} )$.  A calculation gives
\[
(1 + \varpi^l A)(Y + y)(1 + \varpi^l A)^{-1} \equiv Y + y + \varpi^l[A, Y] \; (\p^{l+1}).
\]
Because $Y$ is regular, $[A, Y]$ can be made equal to any matrix in $\cL$ with zeros on the diagonal if $A$ is chosen correctly.  We may therefore choose $A$ so that the off-diagonal entries of $Y + y + \varpi^l[A, Y]$ all lie in $\p^{l+1}$, which means that it may be written in the form $Z + z$ as in the statement of the lemma.

\end{proof}

\subsection{Compatible pairs}
\label{sec:compatpairs}

We now define the notion of compatible pairs of $\chi$-types, and microlocal lift vectors, for a GGP pair of general linear groups.  We let $G = {\rm GL}_{n+1}(F)$ and $H = {\rm GL}_n(F)$, and consider $H$ to be embedded in the upper left corner of $G$.  We continue to use the notation of the previous sections for $G$, and indicate the corresponding objects for $H$ by adding a subscript.  We fix a generic pair $(\chi, \chi_H)$ of conductor $\p^{2l}$.

\begin{definition}
\label{compatdef}

We say that a $\chi$-type $(J, \tl)$ and a $\chi_H$-type $(J_H, \tl_H)$ are compatible if $\tl$ and $\tl_H$ agree on $K_H(l)$.  We say that nonzero microlocal lift vectors $v \in I(\chi)$ and $v_H \in I(\chi_H)$ are compatible if their types are, or equivalently, if they transform by the same character under $K_H(l)$.

\end{definition}

For simplicity, we will refer to $(J, \tl)$ and $(J_H, \tl_H)$ as compatible $(\chi, \chi_H)$-types.  If $\pi : \cL / \p^l \cL \to \cL_H / \p^l \cL_H$ is the natural projection map, compatibility of $(J, \tl)$ and $(J_H, \tl_H)$ is equivalent to $\pi( \exp^*(\widetilde{\lambda})) = \exp^*_H( \widetilde{\lambda}_H)$.  It is clear that $K_H$ acts by conjugation on the set of compatible $(\chi, \chi_H)$-types.

\begin{prop}
\label{compat}

If $(J_H, \tl_H)$ is a $\chi_H$-type, there exists a $\chi$-type $(J, \tl)$ compatible with it, and $J_H$ acts transitively on the set of such $\chi$-types.

\end{prop}

It follows from this that $K_H$ acts transitively on the set of compatible $(\chi, \chi_H)$-types.

\begin{proof}

Because $K_H$ acts transitively on $\chi_H$-types, for existence, it suffices to prove that there exists a $\chi$-type compatible with $(\widetilde{T}_H(l), \tc_H)$, and for transitivity, it suffices to show that $\widetilde{T}_H(l)$ acts transitively on the set of such $\chi$-types.

Consider a $\chi$-type $(J, \tl) = (\widetilde{T}(l)^k, \tl^k)$ for some $k \in K$.  Because $\exp^*(\tl) = D(\chi)^k$, $(J, \tl)$ is compatible with $(\widetilde{T}_H(l), \tc_H)$ if and only if $\pi( D(\chi)^k) = D(\chi_H)$.  To explicate this, we write $D(\chi) = \text{diag}( \alpha_1, \ldots, \alpha_{n+1})$ and $D(\chi_H) = \text{diag}(\beta_1, \ldots, \beta_n)$, where the $\alpha_i$ and $\beta_j$ are mutually distinct mod $\p$ by the comments of Section \ref{sec:characters}.  Compatibility is then equivalent to
\be
k \left( \begin{array}{ccc} \alpha_1 && \\ & \ddots & \\ && \alpha_{n+1} \end{array} \right) k^{-1} =
\left( \begin{array}{cccc}
\beta_1 &&& x_1 \\
& \ddots && \vdots \\
&& \beta_n & x_n \\
y_1 & \ldots & y_n & z
\end{array} \right)
\ee
for some $x_i, y_i, z \in \cO / \p^l$.  Lemma \ref{adjproj} says that there exists a $k_0$ with this property, and moreover that $k \in K$ has this property if and only if $k \in \widetilde{T}_H(l) k_0 \widetilde{T}(l)$.  This implies the proposition.

\end{proof}

\begin{lemma}
\label{adjproj}

Let $A$ be the diagonal subgroup of ${\rm GL}_{n+1}(\cO / \p^l)$, and let $A_H \subset A$ be the subgroup of elements whose last entry is equal to 1.  Let $\alpha_i \in \cO/\p^l$, $1 \leqslant i \leqslant n+1$, and $\beta_j \in \cO/\p^l$, $1 \leqslant j \leqslant n$, have mutually distinct reductions modulo $\p$.  Then $g_0 \in {\rm GL}_{n+1}(\cO / \p^l)$ has the property that
\be
\label{conjugate}
g_0 \left( \begin{array}{ccc} \alpha_1 && \\ & \ddots & \\ && \alpha_{n+1} \end{array} \right) g_0^{-1} =
\left( \begin{array}{cccc}
\beta_1 &&& x_1 \\
& \ddots && \vdots \\
&& \beta_n & x_n \\
y_1 & \ldots & y_n & z
\end{array} \right)
\ee
for some $x_i, y_i, z \in \cO / \p^l$ if and only if
\be
\label{g0}
g_0 \in A_H
\left( \begin{array}{ccc}
& \left( \frac{1}{ \alpha_j - \beta_i } \right)_{ij} \\
1 & \ldots & 1
\end{array} \right)
A.
\ee

\end{lemma}

\begin{proof}

We first show that (\ref{conjugate}) implies (\ref{g0}).  If we write $\alpha' = \text{diag}(\alpha_1, \ldots, \alpha_n)$ and $\beta = \text{diag}(\beta_1, \ldots, \beta_n)$, we may write the equation (\ref{conjugate}) as
\[
g_0 \left( \begin{array}{cc} \alpha' & \\ & \alpha_{n+1} \end{array} \right) = \left( \begin{array}{cc} \beta & x \\ {}^t y & z \end{array} \right) g_0.
\]
If we write $g_0$ in the form
\[
g_0 = \left( \begin{array}{cc} A & b_1 \\ {}^t b_2 & c \end{array} \right),
\]
where $A \in M_n(\cO / \p^l)$ and $b_1$ and $b_2$ are column vectors, then this becomes
\be
\label{conjugate1}
\left( \begin{array}{cc} A \alpha' & \alpha_{n+1} b_1 \\ {}^t b_2 \alpha' & c \alpha_{n+1} \end{array} \right) = \left( \begin{array}{cc} \beta A + x {}^t b_2 & \beta b_1 + cx \\ {}^t y A + z {}^t b_2  & {}^t y b_1 + cz \end{array} \right).
\ee
The top left entry gives
\begin{align*}
A_{ij} \alpha_j & = \beta_i A_{ij} + x_i b_{2,j} \\
A_{ij} & = \frac{ x_i b_{2,j} }{ \alpha_j - \beta_i}
\end{align*}
for all $i, j$, and the top right entry gives
\begin{align*}
\alpha_{n+1} b_{1,i} & = \beta_i b_{1,i} + c x_i \\
b_{1,i} & = \frac{c x_i}{\alpha_{n+1} - \beta_i }
\end{align*}
for all $i$.  Combining these gives that any $g_0$ satisfying (\ref{conjugate}) must have the form
\[
g_0 = 
\left( \begin{array}{cccc}
x_1 &&& \\
& \ddots && \\
&& x_n & \\
&&& 1
\end{array} \right)
\left( \begin{array}{ccc}
& \left( \frac{1}{ \alpha_j - \beta_i } \right)_{ij} \\
1 & \ldots & 1
\end{array} \right)
\left( \begin{array}{cccc}
b_{2,1} &&& \\
& \ddots && \\
&& b_{2,n} & \\
&&& c
\end{array} \right),
\]
which is equivalent to (\ref{g0}) because all the matrices on the right hand side must be invertible.

We next show that (\ref{g0}) implies (\ref{conjugate}).  Because the set of $g_0$ satisfying (\ref{conjugate}) is bi-invariant under $A_H$ and $A$, it suffices to check the case when
\[
g_0 =
\left( \begin{array}{ccc}
& \left( \frac{1}{ \alpha_j - \beta_i } \right)_{ij} \\
1 & \ldots & 1
\end{array} \right).
\]
We shall show that there exists $y_i$ and $z$ such that
\be
\label{conjugate2}
g_0 \left( \begin{array}{cc} \alpha' & \\ & \alpha_{n+1} \end{array} \right) = \left( \begin{array}{cc} \beta & 1_n \\ {}^t y & z \end{array} \right) g_0,
\ee
where $1_n$ denotes the column vector of 1's.  If we write $g_0$ as
\[
g_0 = \left( \begin{array}{cc} A & b_1 \\ {}^t 1_n & 1 \end{array} \right),
\]
then (\ref{conjugate2}) becomes
\[
\left( \begin{array}{cc} A \alpha' & \alpha_{n+1} b_1 \\ {}^t 1_n \alpha' & \alpha_{n+1} \end{array} \right) = \left( \begin{array}{cc} \beta A + I_n & \beta b_1 + 1_n \\ {}^t y A + z {}^t 1_n  & {}^t y b_1 + z \end{array} \right).
\]
As before, the top left and right entries of this equation are satisfied, and it remains to consider the two equations ${}^t 1_n \alpha' = {}^t y A + z {}^t 1_n$ and $\alpha_{n+1} = {}^t y b_1 + z$.  Writing these out, we obtain the system of linear equations
\be
\label{FanPall}
\alpha_j = z + \sum_{i = 1}^n y_i \frac{1}{\alpha_j - \beta_i}
\ee
for $1 \leqslant j \leqslant n+1$.  This system is considered in the Lemma on p. 300 of \cite{FP}.  It is proved that the determinant\footnote{The authors in \cite{FP} work over the reals, but they prove an identity of polynomials that is valid over any ring.} of this system is equal to
\[
(-1)^{n(n-1)/2} \frac{ \prod_{1 \leqslant i < j \leqslant n+1} (\alpha_i - \alpha_j) \prod_{1 \leqslant i < j \leqslant n} (\beta_i - \beta_j) }{ \prod_{1 \leqslant i \leqslant n+1} \prod_{1 \leqslant j \leqslant n} (\alpha_i - \beta_j) },
\]
which by our assumptions on $\alpha_i$ and $\beta_j$ is invertible.  It follows that there are unique $z$ and $y_i$ that satisfy (\ref{FanPall}), and hence (\ref{conjugate2}), as required.  Moreover, Fan and Pall prove that
\[
y_i = - \frac{ \prod_{j = 1}^{n+1} (\alpha_j - \beta_i) }{ \prod_{j \neq i} ( \beta_j - \beta_i) },
\]
and taking traces gives that $z = \sum \alpha_i - \sum \beta_i$.

\end{proof}

We shall need the following additional property of the element $g_0$.

\begin{lemma}
\label{g0entry}

If $g_0$ is as in Lemma \ref{adjproj}, then all entries of $g_0$ and $g_0^{-1}$ are in $(\cO / \p^l)^\times$.

\end{lemma}

\begin{proof}

The statement for $g_0$ follows immediately from (\ref{g0}).  For $g_0^{-1}$, we note that it satisfies the equation
\[
\left( \begin{array}{cc} \alpha' & \\ & \alpha_{n+1} \end{array} \right) g_0^{-1} = g_0^{-1} \left( \begin{array}{cc} \beta & x \\ {}^t y & z \end{array} \right).
\]
If we write
\[
g_0^{-1} = \left( \begin{array}{cc} A' & b'_1 \\ {}^t b'_2 & c' \end{array} \right),
\]
then we see as in Lemma \ref{adjproj} that
\[
A'_{ij} = \frac{ b'_{1i} y_j}{ \alpha_i - \beta_j}, \quad b'_{2j} = \frac{c' y_j}{ \alpha_{n+1} - \beta_j}.
\]
It follows that
\[
g_0^{-1} = 
\left( \begin{array}{ccccc}
b'_{11} &&& \\
& \ddots && \\
&& b'_{1n} & \\
&&& c' 
\end{array} \right)
\left( \begin{array}{ccc}
& & 1 \\
& \left( \frac{1}{ \alpha_i - \beta_j } \right)_{ij} & \vdots \\
& & 1
\end{array} \right)
\left( \begin{array}{cccc}
y_1 &&& \\
& \ddots && \\
&& y_n & \\
&&& 1
\end{array} \right),
\]
which implies the result.

\end{proof}

\begin{lemma}
\label{Atrans}

Let $(J, \tl)$ and $(J_H, \tl_H)$ be compatible $(\chi, \chi_H)$-types.  Let $v \in I(\chi)$ be a microlocal lift vector associated to $(J, \tl)$, and let $v^\vee \in I(\chi^{-1})$ be a microlocal lift vector associated to $(J, \tl^{-1})$.  Then the intersection of $H$ with the support of the matrix coefficient $\rho(g) = \langle gv, v^\vee \rangle$ is equal to $K_H(l)$.  Moreover, we have $J \cap H = K_H(l)$.

\end{lemma}

\begin{proof}

Assume that $(J_H, \tl_H) = (\widetilde{T}_H(l), \tc_H)$ is standard, so that $(J, \tl) = ( \widetilde{T}(l)^{g_0}, \tc^{g_0})$ with $g_0$ as in Lemma \ref{adjproj}.  Lemma \ref{coeffsupport} then implies that $\text{supp}(\rho)$ is contained in $K(l) T^{g_0} K(l)$.  It therefore suffices to show that
\be
\label{coeffH}
K(l) T^{g_0} K(l) \cap H = K_H(l),
\ee
as this gives $\text{supp}(\rho) \cap H \subset K_H(l)$, and the reverse inclusion $K_H(l) \subset \text{supp}(\rho) \cap H$ is clear.  Moreover, (\ref{coeffH}) implies that $J \cap H = K_H(l)$, as we have
\[
K_H(l) \subset J \cap H \subset K(l) T^{g_0} K(l) \cap H = K_H(l).
\]

To establish (\ref{coeffH}), let $h \in K(l) T^{g_0} K(l) \cap H$, and write $h$ as $k_1 g_0 t g_0^{-1} k_2$ with $k_i \in K(l)$ and $t = \text{diag}(t_1, \ldots, t_{n+1})$.   We let $1 \leqslant j \leqslant n+1$ be such that $| t_j |$ is maximal; by replacing $h$ with $h^{-1}$ if necessary we may also assume that $| t_j | \geqslant | t_i^{-1} |$ for all $i$.  Let $e_i$ be the standard basis for $F^{n+1}$.  We will show that $t_i \in \cO^\times$ for all $i$ by examining the vector
\[
h g_0 e_j = k_1 g_0 t g_0^{-1} k_2 g_0 e_j.
\]
Because $g_0 e_j \in \cO^{n+1}$, we have $k_2 g_0 e_j \in g_0 e_j + \p^l \cO^{n+1}$, so that
\begin{align*}
h g_0 e_j & \in k_1 g_0 t g_0^{-1} (g_0 e_j + \p^l \cO^{n+1}) \\
& = k_1 g_0 (t e_j + t \p^l \cO^{n+1}).
\end{align*}
We have $t e_j = t_j e_j$, while our assumption that $| t_j |$ was maximal implies that $t \p^l \cO^{n+1} \subset t_j \p^l \cO^{n+1}$, which gives
\begin{align*}
h g_0 e_j & \in k_1 g_0 t_j (e_j + \p^l \cO^{n+1}) \\
& = t_j g_0 (e_j + \p^l \cO^{n+1}) \\
& = t_j g_0 e_j + t_j \p^l \cO^{n+1}.
\end{align*}
Because $h \in H$, it preserves the last entry of the vector $g_0 e_j$, which is equal to $(g_0)_{n+1,j}$.  By inspecting the last entry of the vectors in the equation above, we therefore have $(g_0)_{n+1,j} \in t_j (g_0)_{n+1,j} + t_j \p^l \cO$.  Because $(g_0)_{n+1,j} \in \cO^\times$ by Lemma \ref{g0entry}, this implies that $t_j \in \cO$, and therefore that $t_i \in \cO^\times$ for all $i$ by our assumption that $| t_j | \geqslant |t_i|, |t_i^{-1}|$.

Having established that $t_i \in \cO^\times$ for all $i$, we may apply the same argument as above to deduce that $h g_0 e_i \in t_i g_0 e_i + \p^l \cO^{n+1}$ for all $i$.  Comparing the last entries of these vectors as before gives $t_i \in 1 + \p^l$, which implies that $t$, and hence $h$, lie in $K(l)$ as required.

\end{proof}

\section{Proof of the trivial bound}
\label{sec:trivialbd}

\subsection{Notation}
\label{sec:notation1}

If $f$ is a complex-valued function on a group, we denote by $f^\vee$ the function given by $f^\vee(g) = \overline{f}(g^{-1})$.

\subsubsection{Number fields}
\label{sec:number fields}

Let $E/F$ be a CM extension of number fields.  Let $\cO$ and $\A$ be the integers and adeles of $F$.  We will denote places of $F$ by $v$ or $w$, possibly with some extra decoration.  For any place $v$ of $F$, we let $\cO_v$ be the integers in the completion $F_v$, $\p_v$ the maximal ideal of $\cO_v$, $\varpi_v$ a uniformizer, and $q_v$ the order of the residue field.  We fix a place $w$ of $F$ that splits in $E$, and write $\p$ and $q$ for $\p_w$ and $q_w$.

We let $\cO_E$ be the integers of $E$.  We denote places of $E$ by $u$ or $u'$.  If $v$ is a place of $F$, we let $E_v = E \otimes_F F_v$, and likewise for other objects associated with $E$.

\subsubsection{Hermitian spaces}
\label{sec:Hermitian spaces}

Let $V$ be a vector space of dimension $n+1$ over $E$, and let $\langle \, , \, \rangle_V$ be a nondegenerate Hermitian form on $V$ with respect to $E/F$.  We assume that $\langle \, , \, \rangle_V$ is positive definite at all infinite places.  Let $V_H \subset V$ be a codimension one subspace.

Let $x_1, \ldots, x_n$ be a basis of $V_H$, and let $x_{n+1} \in V_H^\perp$ be nonzero, so that $x_1, \ldots, x_{n+1}$ forms a basis of $V$.  We let $L \subset V$ and $L_H \subset V_H$ be the $\cO_E$-lattices spanned by $x_1, \ldots, x_{n+1}$ and $x_1, \ldots, x_n$ respectively.  This definition implies that $L_v = L_{H,v} \oplus \cO_{E,v} x_{n+1}$ for all finite places $v$.  We assume that $\langle \, , \, \rangle_V$ is integral on $L$.  We let $S$ be a finite set of places containing all infinite places and all places that ramify in $E/F$.  We also assume that for $v \notin S$ the lattices $L_v$ and $L_{H,v}$ are self-dual.

We now suppose that $v$ splits in $E/F$ as $u u'$, so that $V_v \simeq V_u \oplus V_{u'}$.  The Hermitian form induces a nondegenerate bilinear pairing $B_v : V_u  \times V_{u'} \to F_v$ of vector spaces over $F_v$, and we may write $\langle \, , \, \rangle_V$ in terms of $B_v$ as
\[
\langle (v_1, v_1'), (v_2, v_2') \rangle_V = ( B_v( v_1, v_2'), B_v( v_2, v_1') ) \in F_v \oplus F_v \simeq E_v.
\]
There is an isomorphism $\iota_{V,v}: V_v \simeq F_v^{n+1} \oplus F_v^{n+1}$ that respects the direct sum decomposition $V_v \simeq V_u \oplus V_{u'}$ and that carries $B_v$ to the standard bilinear form on $F_v^{n+1} \times F_v^{n+1}$.  We may assume that $\iota_{V,v}(V_{H,v}) = F_v^n \oplus F_v^n$.

We now further assume that the split place $v$ does not lie in $S$, and in particular that it is finite.  Because $\cO_{E,v} \simeq \cO_{E,u} \oplus \cO_{E,u'}$, we have $L_v \simeq L_u \oplus L_{u'}$.  Moreover, our assumption that $L_v$ was self-dual implies that $L_u$ and $L_{u'}$ are dual to each other under $B_v$, and likewise for $L_{H,u}$ and $L_{H,u'}$.  It follows that we may choose $\iota_{V,v}$ to send $L_u$ and $L_{u'}$ to the relevant copies of $\cO_v^{n+1} \subset F_v^{n+1}$, and likewise for $L_{H,u}$ and $L_{H,u'}$.

\subsubsection{Algebraic groups}
\label{sec:algebraic groups}

Let $G$ and $H$ be the unitary groups of $V$ and $V_H$. Our assumption that $V$ was positive definite implies that the adelic quotients of $G$ and $H$ are compact.  We let $Z \simeq {\rm U}(1)$ be the center of $G$, and define $\widetilde{H} = ZH$.

If $v$ splits in $E/F$ as $u u'$, the isomorphism $V_v \simeq V_u \oplus V_{u'}$ induces isomorphisms of $G_v$ with ${\rm GL}(V_u)$ and ${\rm GL}(V_{u'})$.  Applying the isomorphism $\iota_{V,v}$ defined above then gives an isomorphism $\iota_v : G_v \simeq {\rm GL}_{n+1}(F_v)$.  Note that this requires us to choose one of the places $u$ and $u'$, and changing our choice has the effect of composing $\iota_v$ with the automorphism $g \mapsto {}^t g^{-1}$ of ${\rm GL}_{n+1}(F_v)$.  We have $\iota_v(H_v) = {\rm GL}_n(F_v)$, embedded in ${\rm GL}_{n+1}$ as the upper left-hand block.

\subsubsection{Compact subgroups}
\label{sec:compact subgroups}

For $v \notin S$, we define $K_v$ and $K_{H,v}$ to be the stabilizer of $L_v$ in $G_v$ and of $L_{H,v}$ in $H_v$ respectively.  Our self-duality assumption on $L_v$ and $L_{H,v}$ implies that $K_v$ and $K_{H,v}$ are hyperspecial subgroups.  Moreover, the relation $L_v = L_{H,v} \oplus \cO_{E,v} x_{n+1}$ implies that $K_{H,v} = H_v \cap K_v$.  When $v \notin S$ is split in $E/F$, it follows from Subsections \ref{sec:Hermitian spaces} and \ref{sec:algebraic groups} that the isomorphism $\iota_v$ sends $K_v$ and $K_{H,v}$ to ${\rm GL}_{n+1}(\cO_v)$ and ${\rm GL}_{n}(\cO_v)$, respectively.  We let $K_w(l)$ and $K_{H,w}(l)$ denote the usual principal congruence subgroups of level $\p^l$.

At the place $w$, we define the maximal compact subgroup $K_{\widetilde{H},w}$ of $\widetilde{H}_w$ to be the product of $K_{H,w}$ and $\cO_w^\times$, where the latter is identified with the maximal compact subgroup of $Z_w \simeq F_w^\times$.

For finite places $v \in S$, we choose compact open subgroups $K_v < G_v$ and $K_{H,v} < H_v$ that stabilize $L_v$ and $L_{H,v}$, and satisfy $K_{H,v} = H_v \cap K_v$.  We let $K_f = \prod_{v < \infty} K_v$, which is compact and open in $G(\A_f)$.  We let $K_\infty = G_\infty$, and put $K = K_\infty K_f$.  We define $K_{H,f}$, $K_{H, \infty}$, and $K_H$ analogously for $H$.

\subsubsection{Measures}
\label{sec:measures}

We choose Haar measures $dg = \prod_v dg_v$ and $dh = \prod_v dh_v$ on $G(\A)$ and $H(\A)$ as follows.  At infinity, we let $dg_\infty$ and $dh_\infty$ give volume 1 to $G_\infty$ and $H_\infty$, and at finite places $v$, we let $dg_v$ and $dh_v$ give volume 1 to the compact subgroups $K_v$ and $K_{H,v}$.  We choose a Haar measure $dz = \prod_v dz_v$ on $Z(\A)$ by requiring that $dz_v$ assign volume 1 to the maximal compact subgroup of $Z_v$ at all places.  We have $\widetilde{H} \simeq H \times Z$, and we equip it with the product measure $dh dz$.

\subsection{A relative trace inequality}

The proof of both the trivial bound Proposition \ref{trivialbound}, and of Theorem \ref{mainperiod}, begin with the following inequality.  Number theorists may view it as the result of dropping all but one term from the spectral side of the relative trace formula for $(G,H)$, while analysis may view it as an application of the $T T^*$ method for bounding the norm of an operator.  In any case, it is proved using an elementary application of Cauchy-Schwartz.

\begin{lemma}
\label{ampineq}

Let $k_0 \in C_c( G(\A))$, and let $k = k_0 * k_0^\vee$.  If $f \in L^2( [G])$ with $\| f \|_2 = 1$, and $f_H \in L^2( [H])$, we have
\[
| \cP( R(k_0) f, \overline{f}_H ) |^2 \leqslant \int_{ [H] \times [H]} \overline{f}_H(x) f_H(y) \sum_{\gamma \in G(F) } k(x ^{-1} \gamma y) dx dy,
\]
where $R(k_0)$ denotes the action of $k_0$ in the right-regular representation.  In particular, the right hand side is non-negative.

\end{lemma}

\begin{proof}

Substituting the definition of $R(k_0) f$ in $\cP$ gives
\[
\cP( R(k_0) f, \overline{f}_H ) = \int_{[H]} \overline{f}_H(x) \int_{G(\A)} k_0(g) f(xg) dg dx,
\]
and after a change of variable this becomes
\[
\cP( R(k_0) f, \overline{f}_H ) = \int_{[H]} \overline{f}_H(x) \int_{G(\A)} k_0(x^{-1} g) f(g) dg dx.
\]
Folding over $G(F)$ and bringing the integral over $[G]$ to the outside gives
\begin{align*}
\cP( R(k_0) f, \overline{f}_H ) & =  \int_{[H]} \overline{f}_H(x) \int_{[G]} \sum_{\gamma \in G(F)} k_0(x^{-1} \gamma g) f(g) dg dx \\
& = \int_{[G]} \int_{[H]} \sum_{\gamma \in G(F)} \overline{f}_H(x) k_0(x^{-1} \gamma g) f(g) dx dg.
\end{align*}
If we apply Cauchy-Schwartz to the integral over $[G]$ and use the fact that $\| f \|_2 = 1$, we obtain
\begin{align*}
| \cP( R(k_0) f, \overline{f}_H ) |^2 & \leqslant \int_{[G]} \left| \int_{[H]} \sum_{\gamma \in G(F)} \overline{f}_H(x) k_0(x^{-1} \gamma g) dx \right|^2 dg \\
& = \int_{[G]} \int_{ [H] \times [H]} \sum_{\gamma_1, \gamma_2 \in G(F)} \overline{f}_H(x) f_H(y) k_0( x^{-1} \gamma_1 g) \overline{k}_0( y^{-1} \gamma_2 g) dx dy dg.
\end{align*}
We now return the integral over $[G]$ to the inside and unfold, which gives
\begin{align*}
| \cP( R(k_0) f, \overline{f}_H ) |^2 & \leqslant \int_{ [H] \times [H]} \overline{f}_H(x) f_H(y) \int_{[G]} \sum_{\gamma_1, \gamma_2 \in G(F)} k_0( x^{-1} \gamma_1 g) \overline{k}_0( y^{-1} \gamma_2 g) dg dx dy \\
& = \int_{ [H] \times [H]} \overline{f}_H(x) f_H(y) \int_{G(\A)} \sum_{\gamma \in G(F)} k_0( x^{-1} \gamma g) \overline{k}_0( y^{-1} g) dg dx dy.
\end{align*}
We may write $\overline{k}_0( y^{-1} g)$ as $k_0^\vee( g^{-1} y)$, so that the integral over $G(\A)$ is
\[
\int_{G(\A)} k_0( x^{-1} \gamma g) \overline{k}_0( y^{-1} g) dg = \int_{G(\A)} k_0( x^{-1} \gamma g) k_0^\vee( g^{-1} y) dg = k(x^{-1} \gamma y).
\]
This gives
\[
| \cP( R(k_0) f, \overline{f}_H ) |^2 \leqslant \int_{ [H] \times [H]} \overline{f}_H(x) f_H(y) \sum_{\gamma \in G(F)} k(x^{-1} \gamma y) dx dy
\]
as required.

\end{proof}

\subsection{Proof of Proposition \ref{trivialbound}}
\label{sec:trivial proof}

We recall the notation of Proposition \ref{trivialbound}, including the generic pair $(\chi, \chi_H)$, the compatible $(\chi, \chi_H)$-types $(J, \tl)$ and $(J_H, \tl_H)$, the functions $\phi$ and $\phi_H$, and the representation $\mu$.  We shall prove Proposition \ref{trivialbound} by applying Lemma \ref{ampineq} with $f$ equal to $\phi$, and $f_H$ running over a decomposition of $\phi_H$ into functions with small support.

We choose the test function $k_0$ to be the following spectral projector onto $\phi$.  Let $v' \neq w$ be an auxiliary finite place.  At infinity, we take $k_{0,\infty}$ to be $d_\mu \chi_\mu$, where $d_\mu$ is the dimension of $\mu$ and $\chi_\mu$ is its character, which projects to the $\mu$-isotypic subspace of $L^2(G_\infty)$.  At $w$, $k_{0,w}$ is the function supported on $J$ and equal to $\text{vol}(J)^{-1} \widetilde{\lambda}^{-1} \sim q^{n(n+1) l} \widetilde{\lambda}^{-1}$ there.  At $v'$, we let $K'_{v'} \subset K_{v'}$ be an open subgroup to be chosen later, and let $k_{0,v'} = \text{vol}(K'_{v'})^{-1} 1_{K'_{v'}}$ be the normalized characteristic function of $K'_{v'}$.  For the remaining places, we let $k_{0,v} = \text{vol}(K_v)^{-1} 1_{K_v}$.  It follows that $k_0$ is a projection operator, i.e. $k_0 = k_0^\vee$ and $k_0 = k_0 * k_0^\vee$.  The transformation properties of $\phi$ imply that $R(k_0) \phi = \phi$.  We define $K' = K'_{v'} \times \prod_{v \neq v'} K_v$, which contains the support of $k_0$.

Applying Lemma \ref{ampineq} with $f = \phi$ and this choice of $k_0$ gives
\be
\label{phiampineq}
| \cP( \phi, \overline{f}_H ) |^2 \leqslant \int_{ [H] \times [H]} \overline{f}_H(x) f_H(y) \sum_{\gamma \in G(F) } k_0(x ^{-1} \gamma y) dx dy
\ee
for any $f_H$.  To choose $f_H$, let $B = K' \cap H(\A)$, and let $x_i B$ be a set of cosets that covers $[H]$.  We let $f_{H,i}$ be a collection of functions such that $\text{supp}(f_i) \subset x_i B$, $\| f_{H,i} \|_2 \leqslant 1$, and $\phi_H = \sum f_{H,i}$.  We have
\[
| \cP( \phi, \overline{\phi}_H) |^2 = \Big| \sum_i \cP( \phi, \overline{f}_{H,i} ) \Big|^2 \ll \sum_i | \cP( \phi, \overline{f}_{H,i} ) |^2.
\]
We may apply (\ref{phiampineq}) to each $f_{H,i}$, and so it suffices to prove that
\be
\label{fHi}
\int_{ [H] \times [H]} \overline{f}_{H,i}(x) f_{H,i}(y) \sum_{\gamma \in G(F) } k_0(x ^{-1} \gamma y) dx dy \ll d_\mu^2 q^{nl}
\ee
for all $i$.

\begin{lemma}

If $K'_{v'}$ is chosen small enough, then only $\gamma \in H(F)$ contribute to the left hand side of (\ref{fHi}).

\end{lemma}

\begin{proof}

We must show that there is a choice of $K'_{v'}$ such that if $x, y \in x_i B$ for some $x_i$, and if $\gamma \in G(F)$ satisfies $x^{-1} \gamma y \in \text{supp}(k_0)$, then $\gamma \in H(F)$.  We are free to translate $x_i$ on the left by $H(F)$, and so we may choose a compact set $\Omega = \prod \Omega_v \subset H(\A)$ containing a fundamental domain for $[H]$ and assume that $x_i \in \Omega$.  If we let $x = x_i b_1$ and $y = x_i b_2$ with $b_1, b_2 \in B$, then we have $b_1^{-1} x_i^{-1} \gamma x_i b_2 \in \text{supp}(k_0)$, so $x_i^{-1} \gamma x_i \in B \text{supp}(k_0) B^{-1} \subset K'$.  This implies that 
\[
\gamma \in x_i K' x_i^{-1} \subset \prod_{v \neq v'} \Omega_v K_v \Omega_v^{-1} \times \bigcup_{y \in \Omega_{v'}} y K'_{v'} y^{-1}.
\]
The compact set above is fixed at places away from $v'$, while at $v'$ it may be made arbitrarily small by shrinking $K'_{v'}$.  It follows that there is a choice of $K'_{v'}$ such that the only $\gamma$ lying in this set is the identity, as required.

\end{proof}

With this choice of $K'_{v'}$, the left hand side of (\ref{fHi}) simplifies to
\bes
\int_{ [H] \times [H]} \overline{f}_{H,i}(x) f_{H,i}(y) \sum_{\gamma \in H(F) } k_0(x ^{-1} \gamma y) dx dy = \langle R_H( k_H ) f_{H,i}, f_{H,i} \rangle,
\ees
where $R_H$ denotes the right-regular representation on $H$ and we have written $k_H$ for the restriction of $k_0$ to $H(\A)$.  We must show that this is $\ll d_\mu^2 q^{nl}$.  We have the trivial bound $\langle R_H( k_H) f_{H,i}, f_{H,i} \rangle \leqslant \| f_{H,i} \|_2^2 \| k_H \|_1 \leqslant \| k_H \|_1$, so it suffices to show $\| k_H \|_1 \ll d_\mu^2 q^{nl}$.  We have $\| k_H \|_1 = \| k_{H,w} \|_1 \| k_H^w \|_1$, and moreover $\| k_H^w \|_1 \ll \| k_H^w \|_\infty \ll d_\mu^2$.  The bound $\| k_{H,w} \|_1 \ll q^{nl}$ follows from the fact that $k_w$ is supported on $J$ and has size $q^{n(n+1)l}$ there, and the transversality result $J \cap H_w = K_{H,w}(\p^l)$ from Lemma \ref{Atrans}.

\section{Prelimiaries to amplification}
\label{sec:amp prelim}

This section contains results that will be used in the amplification argument in Section \ref{sec:amp}.

\subsection{Notation}

Let $\cP$ be the set of places of $F$ that are split in $E/F$ and do not lie in $S$.  Our amplifier will be supported at places in $\cP$.

\subsubsection{Root systems}

For each $v \in \cP$, we have the identification of $G_v$ with ${\rm GL}_{n+1}(F_v)$ chosen in Section \ref{sec:algebraic groups}.  We let $T_v$ be the maximal torus in $G_v$ that corresponds to the diagonal subgroup under this identification, and let $B_v$ correspond to the standard upper triangular Borel.  We let $X^*(T_v)$ and $X_*(T_v)$ be the groups of characters and cocharacters of $T_v$.  We let $\Phi$ be the roots of $T_v$ in $\text{Lie}(G_v)$, with positive roots $\Phi^+$ and simple roots $\Delta$ corresponding to $B_v$.  (Technically these depend on $v$, but we may naturally identify them for all $v$.)  We let $\rho \in X^*(T_v) \otimes_\Z \Q$ denote the half-sum of the positive roots, and $W$ be the Weyl group.  We define
\[
X^+_*(T_v) = \{ \mu \in X_*(T_v) : \langle \mu, \alpha \rangle \geqslant 0, \, \alpha \in \Delta \}.
\]
We introduce the function
\[
\| \mu \|^* = \underset{ w \in W}{\max} \langle w\mu, \rho \rangle
\]
on $X_*(T_v)$.  One sees that $\| \mu \|^*$ is a seminorm, with kernel equal to the central cocharacters; the condition that $\| \mu \|^* = \| -\mu \|^*$ follows from the fact that $\rho$ and $-\rho$ lie in the same Weyl orbit.

We define the analogous objects for $H$ and $\widetilde{H}$, and denote them by adding the appropriate subscript.

\subsubsection{Metrics on groups}
\label{sec:metrics}

Let $u$ be a place of $E$, and let $A \in \End_{E_u}(V_u)$.  If $u$ is infinite, then $V_u$ is a complex vector space with the positive definite Hermitian form $\langle \, , \, \rangle_V$, and we define $\| A \|_u$ to be the usual operator norm of $A$.  If $u$ is finite, we define $\| A \|_u$ to be the maximum of the $u$-adic norms of the matrix entries of $A$ in the basis $x_1, \ldots, x_{n+1}$, or equivalently as the smallest value of $\| a \|_u$ among $a \in E_u$ such that $A L_u \subset a L_u$.  If $A \in \End_E(V)$, we define $\| A \| = \prod_u \| A \|_u$.

If $u$ is a place of $E$ above a place $v$ of $F$, and $g \in G_v$, we may naturally talk about $\| g \|_u$.  It follows from the definition that $\| \cdot \|_u$ is bi-invariant under $K_v$.

Recall the subgroup $K_{\widetilde{H},w}$ of $\widetilde{H}_w$ defined in Section \ref{sec:compact subgroups}.  For $g_w \in K_w$, we define $d_w(g_w, K_{\widetilde{H},w})$ to be 0 if $g_w \in K_{\widetilde{H},w}$, and otherwise to be $q^{-l}$, where $l$ is the largest integer such that $g_w \in K_w(l) K_{\widetilde{H},w}$.

\subsection{Hecke algebras}

We let $\cH$ be the Hecke algebra of compactly supported functions on $G(\A_f)$ that are bi-invariant under $K_f$.  If $S'$ is a finite set of finite places, we likewise define the Hecke algebras $\cH_{S'}$ and $\cH^{S'}$ at $S'$ and away from $S'$, respectively.  For $v \in \cP$ and $\mu \in X_*( T_v)$, define $\tau(v, \mu) \in \cH_v$ to be the function supported on $K_v \mu(\varpi_v) K_v$ and taking the value $q_v^{-\| \mu \|^*}$ there.

We shall use the amplifier for ${\rm GL}_{n+1}$ constructed by Blomer--Maga \cite[Section 4]{BM}.  To introduce this, for $j \in \Z$ let $[j] = (j, 0, \ldots, 0)\in X_*(T_v)$, and $[j,-j] = (j, 0, \ldots, 0, -j)\in X_*(T_v)$.  We shall use the following results from \cite[Section 4]{BM}.

\begin{prop}
\label{amplifier}

Let $v \in \cP$.

\begin{enumerate}[(a)]

\item Let $\pi_v$ be an unramified representation of $G_v$, and $v \in \pi_v$ a nonzero spherical vector.  Define $\lambda(j)$ by $\tau(v, [j]) v = \lambda(j) v$.  If $q_v$ is sufficiently large depending on $n$, there is $1 \leqslant j \leqslant n+1$ such that $\lambda(j) \gg 1$, where the implied constant depends only on $n$.

\item For $1 \leqslant j \leqslant n+1$ we have
\[
\tau(v, [j]) \tau(v, [-j]) = \sum_{i = 0}^j c_{vij} \tau(v, [i,-i]),
\]
where $c_{vij} \ll 1$ and the implied constant depends only on $n$.

\end{enumerate}

\end{prop}

\subsection{Spherical representations and spherical functions}
\label{sec:spherical}

In this section, $v$ will denote a place in $\cP$.  We shall use the following parametrization of unramified characters of $T_v$.  If $T_v^c$ is the maximal compact subgroup of $T_v$, then we may identify $T_v / T_v^c$ with $X_*(T_v)$ via the map sending $\mu \in X_*(T_v)$ to $\mu(\varpi_v) T_v^c$.  This lets us identify the group of unramified characters of $T_v$ with the group of complex characters of $X_*(T_v)$, which we denote by $\widehat{X}_*(T_v)$.  We may naturally identify $\widehat{X}_*(T_v)$ with $(\C^\times)^{n+1}$, and if $\alpha \in \widehat{X}_*(T_v)$ we let $(\alpha_1, \ldots, \alpha_{n+1})$ be its coordinates under this identification.  We say that $\alpha$ is $\theta$-tempered if $q_v^{-\theta} \leqslant |\alpha_i | \leqslant q_v^\theta$ for all $i$.

We next recall the classification of irreducible unramified representations of $G_v$ in terms of unramified characters of $T_v$.  If $\chi$ is an unramified character of $T_v$ corresponding to $\alpha \in \widehat{X}_*(T_v)$, we let $\pi_{v,\alpha}$, or $\pi_{v,\chi}$, be the unique irreducible unramified subquotient of the unitarily induced representation $\text{Ind}_{B_v}^{G_v} \chi$.  It is known \cite[Section 4.4]{Cartier} that all irreducible unramified representations of $G_v$ arise in this way, and that $\pi_{v,\alpha} \simeq \pi_{v,\alpha'}$ if and only if $\alpha$ and $\alpha'$ lie in the same Weyl orbit.  It follows that an irreducible unramified representation $\pi_v$ is isomorphic to $\pi_{v, \alpha}$ for a unique $\alpha \in \widehat{X}_*(T_v) / W$, which we refer to as the Satake parameter of $\pi_v$.  We say that $\pi_v$ is $\theta$-tempered if its Satake parameter is.  We note that the contragredient $\pi_{v, \alpha}^\vee$ of $\pi_{v, \alpha}$ is isomorphic to $\pi_{v, \alpha^{-1}}$.

For $\alpha \in \widehat{X}_*(T_v)$, we denote the spherical function with Satake parameter $\alpha$ by $\varphi_{v,\alpha}$.  To recall its definition, we let $v \in \pi_{v, \alpha}$ and $v^\vee \in \pi_{v, \alpha}^\vee$ be spherical vectors with $\langle v, v^\vee \rangle = 1$.  We then have
\[
\varphi_{v,\alpha}(g) = \langle \pi_{v, \alpha}(g) v, v^\vee \rangle.
\]
We shall require the following bound for $\varphi_{v,\alpha}$.

\begin{lemma}
\label{phibound}

Let $\alpha \in \widehat{X}_*(T_v)$.

\begin{enumerate}[(i)]

\item
\label{phibound1}
When $\alpha = 1$ is the trivial character, we have
\[
\varphi_{v,1}( \mu(\varpi_v) ) \ll (\| \mu \|^*)^{n(n+1)/2} q_v^{ - \| \mu \|^*}
\]
for $\mu \in X_*(T_v)$.

\item
\label{phibound2}
For general $\alpha$, we have
\[
\varphi_{v,\alpha}( \mu(\varpi_v) ) \leqslant \underset{ w \in W }{ \max} | \alpha( w\mu ) | \varphi_{v,1}( \mu(\varpi_v) ) \ll \underset{ w \in W }{ \max} | \alpha( w\mu ) | (\| \mu \|^*)^{n(n+1)/2} q_v^{ - \| \mu \|^*}
\]
for $\mu \in X_*(T_v)$.

\end{enumerate}

Both implied constants depend only on $n$.

\end{lemma}

\begin{proof}

We may assume without loss of generality that $\mu \in X^+_*(T_v)$.  We prove (\ref{phibound1}) using the formula for $\varphi_{v,1}$ given by Macdonald \cite[Prop. 4.6.1]{Mac}.  Note that Macdonald only proves this when $G_v$ is simply connected, but for a general $G_v$ the formula may be derived in the same way from the formula of Casselman \cite[Thm. 4.2]{Ca}.  Macdonald's formula states that
\[
\varphi_{v,1}( \mu(\varpi_v) ) = q_v^{ - \langle \mu, \rho \rangle} P(\mu, q_v^{-1})
\]
for $\mu \in X^+_*(T_v)$, where $P(\mu, q_v^{-1})$ is a polynomial on $X_*(T_v) \times \C$ that depends only on $n$; in other words, it is a polynomial in $\mu$ whose coefficients are polynomials in $q_v^{-1}$.  Moreover, the degree of $P(\mu, q_v^{-1})$ as a function of $\mu$ is at most $| \Phi^+ | = n(n+1)/2$, which gives (\ref{phibound1}).

We prove (\ref{phibound2}) by comparing the integral representations of $\varphi_{v,\alpha}$ and $\varphi_{v,1}$.  To state these representations, we recall the Iwasawa $A$-coordinate on $G_v$.  For $g \in G_v$, this is the element $A(g) \in X_*(T_v)$ such that $g \in N_v A(g)(\varpi_v) K_v$, where $N_v$ is the unipotent radical of $B_v$.  If we let $\delta \in \widehat{X}_*(T_v)$ be the modular character, we then have
\[
\varphi_{v,\alpha}( \mu(\varpi_v) ) = \int_{K_v} (\delta^{1/2} \alpha)( A( k \mu(\varpi_v)) ) dk.
\]
We may compare these formulas for $\varphi_{v,\alpha}$ and $\varphi_{v,1}$ as follows:
\begin{align*}
\varphi_{v,\alpha}( \mu(\varpi_v) ) & = \int_{K_v} (\delta^{1/2} \alpha)( A( k \mu(\varpi_v)) ) dk \\
& \leqslant \underset{k \in K_v}{\max} \, | \alpha( A( k \mu(\varpi_v)) ) |  \int_{K_v} \delta^{1/2}( A( k \mu(\varpi_v)) ) dk \\
& = \underset{k \in K_v}{\max} \, | \alpha( A( k \mu(\varpi_v)) ) | \varphi_{v,1}( \mu(\varpi_v) ).
\end{align*}
It is known that the set $\{ A( k \mu(\varpi_v)) : k \in K_v \}$ lies in the convex hull of $W \mu$, which we denote $\text{Conv}(W \mu)$.  (For instance, this follows by combining \cite[Prop. 5.4.2]{KP} with the fact that the image of the Satake transform is Weyl-invariant.)  We therefore have
\[
\varphi_{v,\alpha}( \mu(\varpi_v) ) \leqslant \underset{ \lambda \in \text{Conv}(W \mu)}{ \max} | \alpha(\lambda) | \varphi_{v,1}( \mu(\varpi_v) ) = \underset{ w \in W }{ \max} | \alpha( w\mu ) | \varphi_{v,1}( \mu(\varpi_v) ),
\]
which completes the proof.

\end{proof}

\subsection{Restriction of Hecke operators to $\widetilde{H}$}
\label{sec:Hecke restriction}

In this section, we continue to let $v$ denote a place in $\cP$.  If $\tau(v, \mu)|_{\widetilde{H}}$ denotes the restriction of $\tau(v, \mu)$ to $\widetilde{H}$, the following lemma gives a bound for the action of $\tau(v, \mu)|_{\widetilde{H}}$ in an unramified representation of $\widetilde{H}$.  This will be used to bound the diagonal term in the amplified relative trace inequality.

\begin{lemma}
\label{Hecke restriction}

Let $\alpha \in \widehat{X}_*(T_{\widetilde{H},v})$ be $\theta$-tempered, and let $\pi^{\widetilde{H}}_{v, \alpha}$ be the unramified representation of $\widetilde{H}_v$ with Satake parameter $\alpha$.  We assume that the central character of $\pi^{\widetilde{H}}_{v, \alpha}$ is unitary.  If $v$ and $v^\vee$ are spherical vectors in $\pi^{\widetilde{H}}_{v, \alpha}$ and $(\pi^{\widetilde{H}}_{v, \alpha})^\vee$ with $\langle v, v^\vee \rangle = 1$, and $\mu \in X_*(T_v)$ is not a multiple of $(1, \ldots, 1)$, we have
\[
\langle \pi^{\widetilde{H}}_{v, \alpha}( \tau(v, \mu)|_{\widetilde{H}} ) v, v^\vee \rangle \ll q_v^{-1/2 + \theta},
\]
where the implied constant depends only on $n$ and $\theta$.

\end{lemma}

\begin{proof}

We begin by describing the function $\tau(v, \mu)|_{\widetilde{H}}$.  We claim that
\be
\label{coset restriction}
K_v \mu(\varpi_v) K_v \cap \widetilde{H}_v = \bigcup_{\lambda \in W \mu} K_{\widetilde{H},v} \lambda(\varpi_v) K_{\widetilde{H},v}.
\ee
To prove the claim, it is clear that the right hand set is contained in the left.  For the reverse inclusion, let $h \in K_v \mu(\varpi_v) K_v \cap \widetilde{H}_v$.  The Cartan decomposition on $\widetilde{H}_v$ implies that $h \in K_{\widetilde{H},v} \lambda(\varpi_v) K_{\widetilde{H},v}$ for some $\lambda \in X_*(T_{\widetilde{H},v})$, and comparing this with the Cartan decomposition on $G_v$ we see that $\lambda \in W\mu$ as required.

Let $\lambda^{(1)}, \ldots, \lambda^{(a)}$ be a set of representatives for the $W_H$-orbits in $W \mu$.  Equation (\ref{coset restriction}) then gives $\tau(v, \mu)|_{\widetilde{H}} = q_v^{- \| \mu \|^*} \sum_{1 \leqslant i \leqslant a} 1_{\widetilde{H}}(v, \lambda^{(i)})$, where $1_{\widetilde{H}}(v, \lambda)$ denotes the characteristic function of the set $K_{\widetilde{H},v} \lambda(\varpi_v) K_{\widetilde{H},v}$.  This formula for $\tau(v, \mu)|_{\widetilde{H}}$ gives
\begin{align*}
\langle \pi^{\widetilde{H}}_{v, \alpha}( \tau(v, \mu)|_{\widetilde{H}} ) v, v^\vee \rangle & = q_v^{- \| \mu \|^*} \sum_{1 \leqslant i \leqslant a} \int_{ K_{\widetilde{H},v} \lambda^{(i)}(\varpi_v) K_{\widetilde{H},v} } \varphi^{\widetilde{H}}_{v, \alpha}(h) dh \\
& = q_v^{- \| \mu \|^*} \sum_{1 \leqslant i \leqslant a} \text{vol}( K_{\widetilde{H},v} \lambda^{(i)}(\varpi_v) K_{\widetilde{H},v} ) \varphi^{\widetilde{H}}_{v, \alpha}( \lambda^{(i)}(\varpi_v)),
\end{align*}
where $\varphi^{\widetilde{H}}_{v, \alpha}$ is the spherical function on $\widetilde{H}_v$.  We have $\text{vol}( K_{\widetilde{H},v} \lambda^{(i)}(\varpi_v) K_{\widetilde{H},v} ) \ll q_v^{ 2 \| \lambda^{(i)} \|_{\widetilde{H}}^*}$, where the implied constant depends only on $n$.  Combining this with Lemma \ref{phibound} for the group $\widetilde{H}_v$ gives
\be
\label{lambda i sum}
\langle \pi^{\widetilde{H}}_{v, \alpha}( \tau(v, \mu)|_{\widetilde{H}} ) v, v^\vee \rangle \ll q_v^{- \| \mu \|^*} \sum_{1 \leqslant i \leqslant a} \underset{ w \in W_H}{ \max} | \alpha(w \lambda^{(i)}) | ( \| \lambda^{(i)} \|^*_{\widetilde{H}} )^{ n(n-1)/2} q_v^{ \| \lambda^{(i)} \|_{\widetilde{H}}^*}.
\ee

We next examine the relation between $\mu$ and $\lambda^{(i)}$, and the difference $\| \lambda^{(i)} \|_{\widetilde{H}}^* - \| \mu \|^*$.  Write $\mu = (\mu_1, \ldots, \mu_{n+1})$, and assume that the $\mu_i$ are in non-increasing order.  There is a $1 \leqslant k \leqslant n+1$ such that $\lambda^{(i)} = (\mu_1, \ldots, \mu_{k-1}, \mu_{k+1}, \ldots, \mu_{n+1}, \mu_k)$.  Moreover, because all the expressions involving $\mu$ and $\lambda^{(i)}$ appearing on the right hand side of (\ref{lambda i sum}) are invariant under adding multiples of $(1, \ldots, 1)$ to $\mu$, we may assume that $\mu_k = 0$, and therefore that $\mu_1 \geqslant \ldots \geqslant \mu_{k-1} \geqslant 0 \geqslant \mu_{k+1} \geqslant \ldots \geqslant \mu_{n+1}$.  (Note that by assuming our original $\mu$ is not a multiple of $(1, \ldots, 1)$, we ensure that this normalized $\mu$ is not zero.)  It follows that
\[
\| \lambda^{(i)} \|_{\widetilde{H}}^* - \| \mu \|^* = -\frac{1}{2}( \mu_1 + \ldots + \mu_{k-1}) + \frac{1}{2}( \mu_{k+1} + \ldots + \mu_{n+1}) = -\frac{1}{2} \sum_{j=1}^{n+1} | \mu_j |.
\]
Our assumption that $\alpha$ is $\theta$-tempered implies that
\[
\underset{ w \in W_H}{ \max} | \alpha(w \lambda^{(i)}) | \leqslant q_v^{\theta \sum_{j = 1}^{n+1} | \lambda^{(i)}_j |} = q_v^{ \theta \sum_{j = 1}^{n+1} | \mu_j |},
\]
which gives
\begin{align*}
q_v^{- \| \mu \|^*} \underset{ w \in W_H}{ \max} | \alpha(w \lambda^{(i)}) | ( \| \lambda^{(i)} \|^*_{\widetilde{H}} )^{ n(n-1)/2} q_v^{ \| \lambda^{(i)} \|_{\widetilde{H}}^*} & \leqslant ( \| \lambda^{(i)} \|^*_{\widetilde{H}} )^{ n(n-1)/2} q_v^{ (-1/2 + \theta) \sum | \mu_j |} \\
& \ll \Big( \sum_{j=1}^{n+1} | \mu_j | \Big)^{ n(n-1)/2} q_v^{ (-1/2 + \theta) \sum | \mu_j |}.
\end{align*}
When $\sum | \mu_j | = 1$ this is $\ll q_v^{-1/2 + \theta}$, while for $\sum | \mu_j | \geqslant 2$ it is $\ll_\epsilon q_v^{ (-1/2 + \theta + \epsilon) \sum | \mu_j |} \leqslant q_v^{-1/2 + \theta}$, if $\varepsilon$ is chosen small enough depending on $\theta$. Summing this bound over $i$ in (\ref{lambda i sum}) completes the proof.

\end{proof}

\subsection{Diophantine lemmas}
\label{sec:diophantine}

In this section we prove Lemma \ref{Hclose}, which will be used to show that those $\gamma \in G(F)$ that contribute to the off-diagonal term of the geometric side of the relative trace inequality are bounded away from $\widetilde{H}$.  We recall the norms and distance functions introduced in Section \ref{sec:metrics}.

\begin{lemma}
\label{Hecke height}

Let $v \in \cP$ split in $E$ as $v = u u'$.  If $g \in K_v \mu(\varpi_v) K_v$ for $\mu \in X_*(T_v)$, then $\{ \| g \|_u, \| g \|_{u'} \} = \{ q_v^{ \max \mu_i }, q_v^{ \max -\mu_i } \}$.

\end{lemma}

\begin{proof}

Because $\| \cdot \|_u$ and $\| \cdot \|_{u'}$ are bi-invariant under $K_v$, it suffices to calculate $\{ \| \mu(\varpi_v) \|_u, \| \mu(\varpi_v) \|_{u'} \}$.  We may assume without loss of generality that $u$ is the place used to define the isomorphism $\iota_v$ of Section \ref{sec:algebraic groups}.  In this case, it follows from our definitions that $\| \mu(\varpi_v) \|_u = q_v^{ \max \mu_i}$.  Moreover, the discussion in Subsection \ref{sec:algebraic groups} implies that $\| \mu(\varpi_v) \|_{u'}$ is equal to the maximum valuation of the entries of $\mu(\varpi_v)^{-1}$, which is $q_v^{ \max -\mu_i }$.

\end{proof}

\begin{lemma}
\label{Hclose}

Let $\gamma \in G(F)$, and assume that $\gamma_w \in K_w$.  Let $\| \gamma \|^w = \prod_{u \nmid w} \| \gamma \|_u$ be the contribution to $\| \gamma \|$ from places not dividing $w$.  There is $C > 0$ depending only on $V$ such that if
\be
\label{distance vs height}
\| \gamma \|^w d(\gamma_w, K_{\widetilde{H},w})^2 < C
\ee
then $\gamma \in \widetilde{H}(F)$.

\end{lemma}

\begin{proof}

It suffices to prove that (\ref{distance vs height}) implies
\be
\label{V product formula}
\prod_u | \langle \gamma x_{n+1}, x_i \rangle_V |_u < 1
\ee
for all $i \neq n+1$.  Indeed, this inequality implies that $\langle \gamma x_{n+1}, x_i \rangle_V = 0$ for all $i \neq n+1$ by the product formula, and hence that $\gamma x_{n+1} \in Ex_{n+1}$ as required.

To establish (\ref{V product formula}), we bound $| \langle \gamma x_{n+1}, x_i \rangle_V |_u$ at each place $u$.  When $u$ is infinite, we have
\[
| \langle \gamma x_{n+1}, x_i \rangle_V |_u \leqslant \| \gamma \|_u \| x_{n+1} \|_u \| x_i \|_u \leqslant \| x_{n+1} \|_u \| x_i \|_u,
\]
which is bounded by a constant depending only on $V$.  Next, suppose that $v \neq w$ is a finite place of $F$ that splits in $E$ as $v = u u'$.  Let $a_u \in E_u$ be an element of minimal valuation such that $\gamma L_u \subset a_u L_u$, so that $\| \gamma \|_u = \| a_u \|_u$.  By the remarks of Subsection \ref{sec:Hermitian spaces}, the $u$-component $\langle \gamma x_{n+1}, x_i \rangle_{V,u}$ is equal to $B_v( \gamma x_{n+1,u}, x_{i,u'} )$, where $B_v : V_u \times V_{u'} \to F_v \simeq E_u$ is the bilinear pairing introduced there.  Moreover, we have $x_{n+1,u} \in L_u$ and $x_{i,u'} \in L_{u'}$, and our assumption that $\langle \, , \, \rangle_V$ is integral on $L$ means that $B_v$ pairs $L_u$ and $L_{u'}$ integrally.  It follows that $B_v( \gamma x_{n+1,u}, x_{i,u'} ) \in a_u \cO_{E,u}$, and hence that $| \langle \gamma x_{n+1}, x_i \rangle_V |_u \leqslant \| \gamma \|_u$.  The same bound may also be established when $u$ lies over a nonsplit finite place of $F$.

Finally, let $w$ split in $E$ as $u u'$.  Let $d(\gamma_w, K_{\widetilde{H},w}) = q^{-l}$, so that $\gamma_w \in K_{\widetilde{H},w} K_w(l)$.  This implies that $\gamma x_{n+1,u} \in \cO_{E,u} x_{n+1,u} + \p^l L_u$, and it follows as in the case above that $| \langle \gamma x_{n+1}, x_i \rangle_V |_u \leqslant q^{-l} = d(\gamma_w, K_{H,w})$.

Combining these bounds gives
\[
\prod_u | \langle \gamma x_{n+1}, x_i \rangle_V |_u < C_\infty \| \gamma \|^w d(\gamma_w, K_{H,w})^2,
\]
where $C_\infty$ is a constant depending only on $V$ coming from the infinite places.  If we choose the constant $C$ in the lemma to be $1/C_\infty$, then condition (\ref{distance vs height}) implies (\ref{V product formula}) as required.

\end{proof}

\subsection{Temperedness of $\pi$ and $\pi_H$}
\label{sec:tempered}

We now deduce the temperedness of $\pi$ and $\pi_H$ at split places from work of Caraiani and Labesse.

\begin{lemma}
\label{tempered}

The representations $\pi_v$ and $\pi_{H,v}$ are tempered for all $v$ that split in $E/F$.

\end{lemma}

\begin{proof}

It suffices to discuss the case of $\pi_v$.  Labesse \cite[Corollary 5.3]{Lab} shows that $\pi$ admits a weak base change to ${\rm GL}_{n+1}(\A_E)$, which we denote by $\Pi$, that is equal to the isobaric direct sum $\Pi_1 \boxplus \ldots \boxplus \Pi_r$ of discrete, conjugate self-dual representations $\Pi_i$ of ${\rm GL}_{n_i}(\A_E)$.  Moreover, he proves that $\pi$ and $\Pi$ satisfy the usual local compatibility at places that split in $E/F$.  Because $\pi_w$ is an irreducible principal series representation, it is generic, and by local-global compatibility $\Pi_u$ is also generic for $u | w$.  This implies that each of the discrete representations $\Pi_i$ is in fact cuspidal.

The infinitesimal character of $\pi_\infty$ is regular C-algebraic, which by \cite[Corollary 5.3]{Lab} implies that the same is true for $\Pi_\infty$.  If we let $\eta$ be a unitary character of $\A_E^\times / E^\times$ whose components at archimedean places are equal to $z / |z|$, it follows as in \cite[Lemma 6.1]{MS} that for each $i$, either $\Pi_i$ or $\Pi_i \otimes \eta$ has regular C-algebraic infinitesimal character.  We may then apply \cite{Caraiani} to deduce that $\Pi_i$ is tempered at all places for all $i$.  This implies that $\Pi$ is tempered at all places, and hence that $\pi$ is tempered at all split $v$ by local-global compatibility.

\end{proof}

\section{Amplification}
\label{sec:amp}

We now begin the proof of Theorem \ref{mainperiod}.    We recall the notation associated with that theorem, including the representations $\pi$ and $\pi_H$, and the automorphic forms $\phi = \otimes_v \phi_v$ and $\phi_H = \otimes_v \phi_{H,v}$.  Let $(J, \tl)$ and $(J_H, \tl_H)$ be the compatible $(\chi, \chi_H)$-types associated with $\phi$ and $\phi_H$.  We may assume without loss of generality that $\pi_\infty$ is isomorphic to a fixed irreducible representation $\mu$ of $G_\infty$.

It will be convenient to enlarge the period $\cP( \phi, \overline{\phi}_H )$ to the subgroup $\widetilde{H}$.  We therefore define $\phi_{ \widetilde{H}} = \otimes_v \phi_{ \widetilde{H},v}$ to be the unique extension of $\phi_H$ to a function on $[\widetilde{H}]$ that transforms under $Z$ according to the central character of $\phi$.  We also define the enlarged period
\[
\widetilde{\cP}( \phi, \overline{\phi}_{ \widetilde{H}} ) = \int_{[\widetilde{H}]} \phi(h) \overline{\phi}_{ \widetilde{H}}(h) dh.
\]
We have $\widetilde{\cP}( \phi, \overline{\phi}_{ \widetilde{H}} ) = \text{vol}( [Z]) \cP( \phi, \overline{\phi}_H )$.  We have the following variant of Lemma \ref{ampineq}: if $k_0 \in C^\infty_c( G(\A))$, and $k = k_0 * k_0^\vee$, then
\be
\label{amp ineq 2}
\text{vol}( [Z])^2 | \cP( R(k_0) \phi, \overline{\phi}_H ) ) |^2 = | \widetilde{\cP}( R(k_0) \phi, \overline{\phi}_{ \widetilde{H}} ) ) |^2 \leqslant \int_{ [\widetilde{H}] \times [\widetilde{H}]} \overline{\phi}_{ \widetilde{H}}(x) \phi_{ \widetilde{H}}(y) \sum_{\gamma \in G(F) } k(x ^{-1} \gamma y) dx dy.
\ee

For every $v \in \cP \setminus \{ w \}$ and $1 \leqslant j \leqslant n+1$, let $\lambda_\phi( v, j)$ be the eigenvalue of $\tau(v, [j])$ on $\phi$, and let $c(v, j) = \lambda_\phi( v, j) / | \lambda_\phi( v, j)|$, with the convention that $0/0 = 0$.  It follows from Proposition \ref{amplifier} that for all but finitely many $v$ we have $\lambda_\phi( v, j) \gg 1$ for some $j$.

We next define the test functions we shall use in our amplification inequality.  At the place $w$, it will be convenient to work on cosets of the subgroup $K_w(1)$, which we denote by $K_w^1$ for brevity.  We likewise define $K_{H,w}^1 = K_{H,w}(1)$, $J^1 = J \cap K_w^1$, and $J_H^1 = J_H \cap K_{H,w}^1$.  We choose our test function $k_{0,w}$ to be the function supported on $J^1$ and equal to $q^{n(n+1)l} \tl^{-1}$ there.   We choose the archimedean test function $k_{0,\infty} = d_\mu \chi_\mu$ as in Section \ref{sec:trivial proof}.

For $N > 0$ to be chosen later, define $\cP_N = \{ v \in \cP \setminus \{ w \} : N/2 < q_v < N \}$.   For $1 \leqslant j \leqslant n+1$ we let
\[
T_j^0 = \sum_{v \in \cP_N} \overline{c(v, j)} \tau(v, [j]),
\]
which we think of as an element of $\cH^w$.  For each such $j$, we define the test function $k^0_j$ to be the product $k_{0,\infty} k_{0,w} T_j^0$.  We define $T_j = T_j^0 * (T_j^0)^\vee$, and $k_j = k^0_j * (k^0_j)^\vee$, and note that the latter is equal to $k_{0,\infty} k_{0,w} T_j$ up to a nonzero constant depending only on $n$ and $w$.  We also define $T = \sum_{j=1}^{n+1} T_j$ and $k = \sum_{j=1}^{n+1} k_j$.  If we apply (\ref{amp ineq 2}) to each $k^0_j$ and sum over $j$, we obtain
\[
\sum_{j=1}^{n+1} | \cP( R(k^0_j) \phi, \overline{\phi}_H ) ) |^2 \ll \int_{ [\widetilde{H}] \times [\widetilde{H}]} \overline{\phi}_{ \widetilde{H}}(x) \phi_{ \widetilde{H}}(y) \sum_{\gamma \in G(F) } k(x ^{-1} \gamma y) dx dy.
\]
$\phi$ is an eigenvector of each $k^0_j$, and if we let $\widehat{k}^0_j(\phi)$ denote the eigenvalue by which it acts, we have $\widehat{k}^0_j(\phi) \gg N^{1-\epsilon}$ for some $j$ by the remark about $\lambda_\phi( v, j)$ above.  It follows that
\[
\sum_{j=1}^{n+1} | \cP( R(k^0_j) \phi, \overline{\phi}_H ) ) |^2 \gg N^{2-\epsilon} | \cP( \phi, \overline{\phi}_H ) ) |^2,
\]
and hence that
\be
\label{amplified period bd}
N^{2-\epsilon} | \cP( \phi, \overline{\phi}_H ) ) |^2 \ll \int_{ [\widetilde{H}] \times [\widetilde{H}]} \overline{\phi}_{ \widetilde{H}}(x) \phi_{ \widetilde{H}}(y) \sum_{\gamma \in G(F) } k(x ^{-1} \gamma y) dx dy.
\ee
The reminder of the proof involves bounding the right hand side of (\ref{amplified period bd}), which we refer to as the geometric side of the relative trace inequality.  We divide this into the diagonal term
\[
{\rm D}(\phi_{ \widetilde{H}}) = \int_{ [\widetilde{H}] \times [\widetilde{H}]} \overline{\phi}_{ \widetilde{H}}(x) \phi_{ \widetilde{H}}(y) \sum_{\gamma \in \widetilde{H}(F) } k(x ^{-1} \gamma y) dx dy,
\]
and the off-diagonal term
\[
{\rm OD}(\phi_{ \widetilde{H}}) = \int_{ [\widetilde{H}] \times [\widetilde{H}]} \overline{\phi}_{ \widetilde{H}}(x) \phi_{ \widetilde{H}}(y) \sum_{\gamma \in G(F) - \widetilde{H}(F) } k(x ^{-1} \gamma y) dx dy.
\]
We note that ${\rm D}(\phi_{ \widetilde{H}}) = \langle \pi_{\widetilde{H}}( k|_{\widetilde{H}}) \phi_{ \widetilde{H}}, \phi_{ \widetilde{H}} \rangle$, where $\pi_{\widetilde{H}}$ is the automorphic representation of $\widetilde{H}$ extending $\pi_H$.

\subsection{Bounding the diagonal term}
\label{sec:diagonal}

The diagonal term is controlled by the next proposition.

\begin{prop}
\label{diagonal bound}

We have
\be
\label{geometric diagonal}
{\rm D}(\phi_{ \widetilde{H}}) = \langle \pi_{\widetilde{H}}( k|_{\widetilde{H}}) \phi_{ \widetilde{H}}, \phi_{ \widetilde{H}} \rangle \ll_\epsilon N^{1 + \epsilon} d_\mu^2 q^{nl}.
\ee

\end{prop}

\begin{proof}

It suffices to prove the bound for each $k_j$, or equivalently for $k_{0,\infty} k_{0,w} T_j$.  We have
\[
\langle \pi_{\widetilde{H}}( k_{0,\infty} k_{0,w} T_j |_{\widetilde{H}}) \phi_{ \widetilde{H}}, \phi_{ \widetilde{H}} \rangle = \langle \pi_{\widetilde{H}}( k_{0,\infty} |_{\widetilde{H}}) \phi_{ \widetilde{H}, \infty}, \phi_{ \widetilde{H}, \infty} \rangle \langle \pi_{\widetilde{H}}( k_{0,w} |_{\widetilde{H}}) \phi_{ \widetilde{H}, w}, \phi_{ \widetilde{H}, w} \rangle \langle \pi_{\widetilde{H}}( T_j |_{\widetilde{H}}) \phi_{ \widetilde{H}}^{w \infty}, \phi_{ \widetilde{H}}^{w \infty} \rangle.
\]
We have $\langle \pi_{\widetilde{H}}( k_{0,\infty}|_{\widetilde{H}}  ) \phi_{\widetilde{H},\infty}, \phi_{\widetilde{H},\infty} \rangle \ll d_\mu^2$ as in Section \ref{sec:trivial proof}.  The comments at the end of Section \ref{sec:trivial proof} show that $\phi_{\widetilde{H}, w}$ is an eigenvector of $k_{0,w}|_{\widetilde{H}}$ with eigenvalue $\sim q^{nl}$, so that
\[
\langle \pi_{\widetilde{H}}( k_{0,w}|_{\widetilde{H}} ) \phi_{\widetilde{H}, w}, \phi_{\widetilde{H}, w} \rangle \ll  q^{nl}.
\]
It therefore suffices to show that
\be
\label{Tjbound}
\langle \pi_{\widetilde{H}}( T_j |_{\widetilde{H}}) \phi_{ \widetilde{H}}^{w \infty}, \phi_{ \widetilde{H}}^{w \infty} \rangle \ll_\epsilon N^{1 + \epsilon}.
\ee

To do this, we write $T_j$ as a sum of factorizable terms, and apply Lemma \ref{Hecke restriction} to each of them.  Expanding $T_j$ gives
\[
T_j = \sum_{v \in \cP_N} | c(v,j)|^2 \tau(v, [j]) \tau(v, [-j]) + \sum_{v_1 \neq v_2 \in \cP_N} c(v_1, j) \overline{c(v_2, j)} \tau(v_1, [j]) \tau(v_2, [-j]),
\]
which we may write as $T_j = T_j^I + T_j^{II}$.  We use Proposition \ref{amplifier} to simplify $T_j^I$ as
\[
T_j^I = \sum_{v \in \cP_N} \sum_{i = 0}^j c_{vij} | c(v,j)|^2 \tau(v, [i,-i]).
\]
Lemma \ref{Hecke restriction} gives
\be
\label{k1bound}
\left\langle \pi_{\widetilde{H}}\left(  \tau(v, [i,-i]) \big|_{\widetilde{H}} \right) \phi_{\widetilde{H}, v}, \phi_{\widetilde{H}, v} \right\rangle \ll \Bigg\{ \begin{array}{ll} 1, & i = 0, \\  q_v^{-1/2}, & i \geqslant 1, \end{array}
\ee
and summing this over $v \in \cP_N$ and $i$ gives
\[
\langle \pi_{\widetilde{H}}( T_j^I |_{\widetilde{H}}) \phi_{ \widetilde{H}}^{w \infty}, \phi_{ \widetilde{H}}^{w \infty} \rangle \ll \sum_{v \in \cP_N} 1 + q_v^{-1/2} \ll N.
\]
Likewise, Lemma \ref{Hecke restriction} gives
\[
\left\langle \pi_{\widetilde{H}}\left(  \tau(v_1, [j]) \tau(v_2, [-j]) \big|_{\widetilde{H}} \right) \phi_{ \widetilde{H}}^{w \infty}, \phi_{ \widetilde{H}}^{w \infty} \right\rangle \ll (q_{v_1} q_{v_2})^{-1/2},
\]
and summing over $v_1$ and $v_2$ gives
\[
\langle \pi_{\widetilde{H}}( T_j^{II} |_{\widetilde{H}}) \phi_{ \widetilde{H}}^{w \infty}, \phi_{ \widetilde{H}}^{w \infty} \rangle \ll \left( \sum_{v \in \cP_N} q_v^{-1/2} \right)^2 \ll_\epsilon N^{1 + \epsilon}.
\]
This completes the proof of (\ref{Tjbound}), and hence of the proposition.

\end{proof}

\subsection{Bounding the off-diagonal term}
\label{sec:off-diag}

We next bound the off-diagonal term ${\rm OD}(\phi_{\widetilde{H}})$.  We shall forgo cancellation in the integrals over $[\widetilde{H}]$, and take absolute values everywhere, which gives
\[
{\rm OD}(\phi_{\widetilde{H}}) \leqslant \int_{ [\widetilde{H}] \times [\widetilde{H}]} | \phi_{ \widetilde{H}}(x) \phi_{ \widetilde{H}}(y) | \sum_{\gamma \in G(F) - \widetilde{H}(F) } |k(x ^{-1} \gamma y)| dx dy.
\]
We let $\Omega = \prod_v \Omega_v \subset H(\A)$ be a compact set containing a fundamental domain for $[\widetilde{H}]$.  After enlarging $S$ if necessary, we assume that $\Omega_v = K_{\widetilde{H},v}$ at places $v \notin S$.  We may expand the integrals over $[\widetilde{H}]$ to $\Omega$, which gives
\be
\label{ODsum}
{\rm OD}(\phi_{\widetilde{H}}) \leqslant \int_{ \Omega \times \Omega} | \phi_{ \widetilde{H}}(x) \phi_{ \widetilde{H}}(y) | \sum_{\gamma \in G(F) - \widetilde{H}(F) } |k(x ^{-1} \gamma y)| dx dy = \sum_{\gamma \in G(F) - \widetilde{H}(F) } I(\phi_{\widetilde{H}}, \gamma),
\ee
where
\[
I(\phi_{\widetilde{H}}, \gamma) = \int_{ \Omega \times \Omega} | \phi_{ \widetilde{H}}(x) \phi_{ \widetilde{H}}(y) k(x ^{-1} \gamma y)| dx dy.
\]
The trivial bound for $I(\phi_{\widetilde{H}}, \gamma)$ is $I(\phi_{\widetilde{H}}, \gamma) \ll d_\mu^2 q^{nl} | T(\gamma)|$, and we shall improve this to
\be
\label{Iimproved}
I(\phi_{\widetilde{H}}, \gamma) \ll d_\mu^2 q^{(n-1/2)l} N^{(n+1)/2} | T(\gamma)|;
\ee
note that we will choose $N$ to be small with respect to $q^l$, so this is in fact a strengthening.  This bound is the key ingredient in the proof of Theorem \ref{mainperiod}, as the saving of $q^{-l/2}$ it gives will imply a corresponding saving for ${\rm OD}(\phi_{\widetilde{H}})$, from which Theorem \ref{mainperiod} follows.  We prove (\ref{Iimproved}) by reducing it to a bound for the norm of an integral operator on $L^2(K_{H,w}^1)$, stated as Proposition \ref{offdiagw} below, and which we establish in Section \ref{sec:offdiag}.

To reduce (\ref{Iimproved}) to Proposition \ref{offdiagw}, we first isolate the local behavior of $I(\phi_{\widetilde{H}}, \gamma)$ at $w$ by introducing an extra integral over $K_{H,w}^1$, which gives
\be
\label{Igamma2}
I(\phi_{\widetilde{H}}, \gamma) =  \text{vol}(K_{H,w}^1)^{-2} \int_{ \Omega \times \Omega} \int_{K_{H,w}^1 \times K_{H,w}^1} | \phi_{ \widetilde{H}}(x h_1) \phi_{ \widetilde{H}}(y h_2) k(h_1^{-1} x^{-1} \gamma y h_2)| dh_1 dh_2 dx dy.
\ee
For $x \in \widetilde{H}(\A)$ we define $f_{H,x}: K_{H,w}^1 \to \C$ by $f_{H,x}(h) = |\phi_{ \widetilde{H}}(xh)|$.  With this notation, we may write the inner integral over $K_{H,w}^1 \times K_{H,w}^1$ as
\begin{multline}
\label{Igamma3}
\int_{K_{H,w}^1 \times K_{H,w}^1} | \phi_{ \widetilde{H}}(x h_1) \phi_{ \widetilde{H}}(y h_2) k(h_1^{-1} x^{-1} \gamma y h_2)| dh_1 dh_2 = \\
 |k^w( x^{-1} \gamma y )| \int_{K_{H,w}^1 \times K_{H,w}^1} f_{H,x}( h_1) f_{H,y}(h_2) |k_w(h_1^{-1} x^{-1} \gamma y h_2)| dh_1 dh_2.
\end{multline}
We note that $| k_w |$ is equal to $q^{n(n+1)l} 1_{J^1}$, up to a constant factor depending only on $n$ and $w$, so we may simplify the integral on the RHS above to
\be
\label{KHlocal}
q^{n(n+1)l} \int_{K_{H,w}^1 \times K_{H,w}^1} f_{H,x}( h_1) f_{H,y}(h_2) 1_{J^1}(h_1^{-1} x^{-1} \gamma y h_2) dh_1 dh_2.
\ee

We shall bound (\ref{KHlocal}) by thinking of it as the matrix coefficient of an operator on $L^2(K_{H,w}^1)$.  For any $\delta \in K_w^1$, we define the operator $A_\delta$ on $L^2(K_{H,w}^1)$ by
\[
A_\delta f(x) = q^{n(n+1)l} \int_{K_{H,w}^1} 1_{J^1}( x^{-1} \delta y) f(y) dy.
\]
Because $J^1 \subset K_w^1$, we may assume that $x^{-1} \gamma y \in K_w^1$ in (\ref{KHlocal}) as the integral vanishes otherwise.  We may therefore rewrite (\ref{KHlocal}) as $\langle A_{x^{-1} \gamma y} f_{H,x}, f_{H,y} \rangle$.  The key to our off-diagonal estimate will be to bound this matrix coefficient under the assumption, provided by Lemma \ref{Hfar} below, that the $\gamma$ making a nonzero contribution to the sum in (\ref{ODsum}) have the property that $\gamma_w$ is bounded away from $K_{\widetilde{H},w}$, and hence that $x^{-1} \gamma y$ is bounded away from $K_{\widetilde{H},w}$ for $x, y \in \Omega$.  We note that we cannot bound $\langle A_{x^{-1} \gamma y} f_{H,x}, f_{H,y} \rangle$ by bounding the operator norm of $A_{x^{-1} \gamma y}$ alone, as the norm of $A_\delta$ will generically be $\gg q^{nl}$, even for $\delta$ that are far from $K_{ \widetilde{H},w}$.  We must also use the information that $\phi_H$ and $f_{H,x}$ are respectively equivariant, and invariant, under $J_H^1$.  We do this by defining $\Pi$ to be the operator on $L^2(K_{H,w}^1)$ given by averaging over $J^1_H$, so that
\[
\Pi f(x) = \text{vol}(J_H^1)^{-1} \int_{J_H^1} f(xg) dg,
\]
and proving the following bound for the norm of $\Pi A_\delta \Pi$.

\begin{prop}
\label{offdiagw}

Let $\delta \in K_w^1$.  The norms of $A_\delta$ and $\Pi A_\delta \Pi$ satisfy the trivial bound $\| A_\delta \| \ll q^{nl}$, and the improvement
\[
\| \Pi A_\delta \Pi \| \ll q^{nl} \min \{ 1, q^{-l/2} d(\delta, K_{\widetilde{H},w})^{-1/2} \}.
\]

\end{prop}

We will combine this with the following lower bound on $d(\gamma_w, K_{\widetilde{H},w})$ for the $\gamma$ contributing to (\ref{ODsum}).

\begin{lemma}
\label{Hfar}

If $\gamma \in G(F) - \widetilde{H}(F)$, and $x, y \in \Omega$ satisfy $k(x^{-1} \gamma y) \neq 0$, then $\gamma_w \in K_w$, and $d(\gamma_w, K_{\widetilde{H},w}) \gg N^{-n-1}$.

\end{lemma}

\begin{proof}

We begin with the assertion that $\gamma_w \in K_w$.  We have $x_w, y_w \in \Omega_w = K_{\widetilde{H},w}$, and $x_w^{-1} \gamma_w y_w \in \text{supp}(k_w) \subset K_w$, so $\gamma_w \in K_w$ as required.

Because $\gamma_w \in K_w$ and $\gamma \notin \widetilde{H}(F)$, we may apply Lemma \ref{Hclose} to deduce that $d(\gamma_w, K_{\widetilde{H},w})^2 \gg 1/\| \gamma \|^w$, where $\| \gamma \|^w = \prod_{u \nmid w} \| \gamma \|_u$.  We must therefore show that $\| \gamma \|^w \ll N^{2n+2}$.  The contribution to $\| \gamma \|^w$ from places $u$ lying above $v \in S$ is bounded, by the condition that $\gamma_v$ lies in the fixed compact set $\Omega_v \text{supp}(k_v) \Omega_v^{-1}$.  We next bound the contribution from places $u$ lying above $v \notin S \cup \{ w \}$ using the fact that $\gamma^S$ must lie in $\text{supp}(T_j)$ for some $1 \leqslant j \leqslant n+1$.  Expanding $T_j$ as in the proof of Proposition \ref{diagonal bound} shows that we either have $\gamma^S \in \text{supp}( \tau(v_1, [i,-i]) )$ for $v_1 \in \cP_N$ and some $0 \leqslant i \leqslant j$, or $\gamma^S \in \text{supp}( \tau(v_1, [j]) \tau(v_2, [j]) )$ for $v_1, v_2 \in \cP_N$.  In the first case we have $\| \gamma \|_u = 1$ unless $u | v_1$, and if $v_1$ splits as $u_1 u'_1$ then Lemma \ref{Hecke height} gives
\[
\| \gamma \|_{u_1} = \| \gamma \|_{u'_1} = q_{v_1}^i < N^{n+1}
\]
so that $\| \gamma \|^w \ll N^{2n+2}$ as required.  We likewise have $\| \gamma \|^w \ll N^{2n+2}$ in the second case, which completes the proof.

\end{proof}

Lemma \ref{Hfar} implies that any $\gamma$ contributing to (\ref{ODsum}) must satisfy $d(\gamma_w, K_{\widetilde{H},w}) \gg N^{-n-1}$, and so the same is true for $x^{-1} \gamma y$ for $x, y \in \Omega$.  We may therefore apply Proposition \ref{offdiagw} with $\delta = x^{-1} \gamma y$ to obtain
\[
\langle A_{x^{-1} \gamma y} f_{H,x}, f_{H,y} \rangle = \langle \Pi A_{x^{-1} \gamma y} \Pi f_{H,x}, f_{H,y} \rangle \ll \| f_{H,x} \|_2 \| f_{H,y} \|_2 q^{(n - 1/2)l} N^{(n+1)/2},
\]
uniformly in $x$, $y$, and $\gamma$.  Substituting this into (\ref{Igamma3}), and then into (\ref{Igamma2}), gives
\be
\label{Ibound2}
I(\phi_{\widetilde{H}}, \gamma) \ll q^{(n - 1/2)l} N^{(n+1)/2} \int_{\Omega \times \Omega} |k^w( x^{-1} \gamma y )|  \| f_{H,x} \|_2 \| f_{H,y} \|_2 dx dy.
\ee
We next observe that $k^w( x^{-1} \gamma y ) \ll d_\mu^2 |T(\gamma)|$, uniformly in $x$ and $y$.  Indeed, we may write $k^w$ as a product $k_S k^{S \cup \{ w \} }$, and have $k_S( x^{-1} \gamma y ) \ll d_\mu^2$, and $k^{S \cup \{ w \} }( x^{-1} \gamma y ) = k^{S \cup \{ w \} }(\gamma) = T(\gamma)$.  Applying this in (\ref{Ibound2}), we have
\begin{align}
\notag
I(\phi_{\widetilde{H}}, \gamma) & \ll d_\mu^2 q^{(n - 1/2)l} N^{(n+1)/2} |T(\gamma)| \int_{\Omega \times \Omega} \| f_{H,x} \|_2 \| f_{H,y} \|_2 dx dy \\
\label{Ibound}
& = d_\mu^2 q^{(n - 1/2)l} N^{(n+1)/2} |T(\gamma)| \left( \int_\Omega \| f_{H,x} \|_2 dx \right)^2.
\end{align}
The Cauchy--Schwartz inequality gives
\[
\left( \int_\Omega \| f_{H,x} \|_2 dx \right)^2 \leqslant \text{vol}(\Omega) \int_\Omega \| f_{H,x} \|_2^2 dx \ll \| \phi_H \|_2^2 = 1,
\]
and applying this in (\ref{Ibound}) gives the required bound (\ref{Iimproved}) for $I(\phi_{\widetilde{H}}, \gamma)$.

With this, we may complete our bound for the geometric side of the trace formula.  Applying (\ref{Iimproved}) in (\ref{ODsum}) gives
\[
{\rm OD}(\phi_{\widetilde{H}}) \ll d_\mu^2 q^{(n - 1/2)l} N^{(n+1)/2} \sum_{\gamma \in G(F) } |T(\gamma)|.
\]
One may easily show that $\sum_{\gamma \in G(F) } |T(\gamma)| \ll \| T \|_1$, and it follows from the definition of $T$ that $\| T \|_1 \ll N^{n^2+n+2}$, so that
\[
{\rm OD}(\phi_{\widetilde{H}}) \ll d_\mu^2 q^{(n - 1/2)l} N^{(2n^2 + 3n+5)/2}.
\]
When combined with the diagonal estimate from Proposition \ref{diagonal bound}, we have our final estimate for the geometric side of the amplified relative trace formula, which is
\[
\int_{ [\widetilde{H}] \times [\widetilde{H}]} \overline{\phi}_{ \widetilde{H}}(x) \phi_{ \widetilde{H}}(y) \sum_{\gamma \in G(F) } k(x ^{-1} \gamma y) dx dy \ll_\epsilon d_\mu^2 q^{nl} N^{1 + \epsilon} + d_\mu^2 q^{(n - 1/2)l} N^{(2n^2 + 3n+5)/2}.
\]
Combining this with (\ref{amplified period bd}) gives
\[
| \cP( \phi, \overline{\phi}_H ) ) |^2 \ll_\epsilon d_\mu^2 q^{nl} N^{-1 + \epsilon} + d_\mu^2 q^{(n - 1/2)l} N^{(2n^2 + 3n+1)/2 + \epsilon}.
\]
Choosing $N = q^{\ell / (2n^2 + 3n + 3)}$ gives Theorem \ref{mainperiod}.

\section{The local off-diagonal bound}
\label{sec:offdiag}

This section contains the proof of Proposition \ref{offdiagw}, which is the key local ingredient in our bound for the off-diagonal terms in Section \ref{sec:amp}.  As we shall work locally at $w$ in this section, we shall omit $w$ from the notation everywhere.

\subsection{The trivial bound}

We begin with the proof of the trivial bound $\| A_\delta \| \ll q^{nl}$.  We derive this from the following standard bound for integral operators on a measure space, see for instance \cite[Thm. 6.18]{Fo}.

\begin{prop}
\label{kernelbound}

Let $K(x,y)$ be a measurable kernel function on a $\sigma$-finite measure space $(X,\mu)$.  Suppose there exists $C > 0$ such that $\int_X | K(x,y) | d\mu(y) \leqslant C$ for a.e. $x \in X$ and $\int_X | K(x,y) | d\mu(x) \leqslant C$ for a.e. $y \in X$, and $1 \leqslant p \leqslant \infty$.  If $f \in L^p(X)$, the integral
\[
Tf(x) = \int K(x,y) f(y) d\mu(y)
\]
converges absolutely for a.e. $x \in X$, the function $Tf$ thus defined is in $L^p(X)$, and $\| Tf \|_p \leqslant C \| f \|_p$.

\end{prop}

If we apply this to the operator $A_\delta$, with integral kernel $q^{n(n+1)l} 1_{J^1}( x^{-1} \delta y)$, it suffices to show that
\be
\label{kernelint}
\underset{x \in K^1_H}{\sup} \int_{K^1_H} 1_{J^1}( x^{-1} \delta y) dy = \underset{x \in K^1_H}{\sup} \text{vol}( K^1_H \cap \delta^{-1} x J^1) \ll q^{-n^2 l},
\ee
and likewise for the $x$-integrals, although we shall omit these as they are similar to the $y$-integrals.  To establish (\ref{kernelint}), let $x \in K^1_H$ be given.  We may suppose that there is some $y_0 \in K^1_H$ such that $y_0 \in \delta^{-1} x J^1$, and then we have
\[
\text{vol}( K^1_H \cap \delta^{-1} x J^1) = \text{vol}( K^1_H \cap y_0^{-1} \delta^{-1} x J^1) = \text{vol}( K^1_H \cap J^1).
\]
We have shown in Lemma \ref{Atrans} that $\text{vol}( K^1_H \cap J^1) \ll q^{-n^2l}$, which gives (\ref{kernelint}) and hence $\| A_\delta \| \ll q^{nl}$.

\subsection{Improving the trivial bound}

We next establish the improved bound
\[
\| \Pi A_\delta \Pi \| \ll q^{nl} \min \{ 1, q^{-l/2} d(\delta, K_{\widetilde{H},w})^{-1/2} \}.
\]
The bound $\| \Pi A_\delta \Pi \| \ll q^{nl}$ follows from $\| A_\delta \| \ll q^{nl}$ and $\| \Pi \| \leqslant 1$, so it suffices to prove
\be
\label{Aimproved}
\| \Pi A_\delta \Pi \| \ll q^{(n-1/2)l} d(\delta, K_{\widetilde{H},w})^{-1/2}.
\ee
We shall establish this by showing that the support of the integral kernel of $A_\delta$ is transverse to $J^1_H$.  In particular, we define
\[
S = \{ (x, y) \in K^1_H \times K^1_H : x^{-1} \delta y \in J^1 \},
\]
which is the support of the function $1_{J^1}( x^{-1} \delta y)$, and define
\begin{align*}
S_1 & = \{ x \in K^1_H :  x^{-1} \delta y \in J^1 \text{ for some } y \in K^1_H \}, \\
S_2 & = \{ y \in K^1_H :  x^{-1} \delta y \in J^1 \text{ for some } x \in K^1_H \},
\end{align*}
to be its projections to the first and second factors. We shall derive (\ref{Aimproved}) from the following transversality statement.

\begin{prop}
\label{transverse}

For any $x, y \in K^1_H$, we have
\[
\textup{vol}(S_1 \cap x J^1_H), \, \textup{vol}(S_2 \cap y J^1_H) \leqslant 2 q^{-l/2+1} d(\delta, K_{\widetilde{H},w})^{-1/2} \textup{vol}(J_H^1).
\]

\end{prop}

The bound (\ref{Aimproved}) follows easily from Proposition \ref{transverse}, as we now demonstrate.  If we let $1_{S_1}$ and $1_{S_2}$ be the multiplication operators by the corresponding characteristic functions, then we are free to insert $1_{S_1}$ and $1_{S_2}$ before and after $A_\delta$ which gives
\[
\| \Pi A_\delta \Pi \| = \| \Pi 1_{S_1} A_\delta 1_{S_2} \Pi \| \leqslant \| \Pi 1_{S_1} \| \| A_\delta \| \| 1_{S_2} \Pi \| \ll q^{nl} \| \Pi 1_{S_1} \| \| 1_{S_2} \Pi \|.
\]
It therefore suffices to show that
\be
\label{Pi1bound}
\| \Pi 1_{S_1} \|, \, \| 1_{S_2} \Pi \| \ll q^{-l/4} d(\delta, K_{\widetilde{H},w})^{-1/4}.
\ee
We shall do this for $\| \Pi 1_{S_1} \|$, as the other case is similar.  If we use the identity $\| T \|^2 = \| T^* T \|$, valid for any bounded operator on a Hilbert space (where $T^*$ denotes the adjoint), we obtain $\| \Pi 1_{S_1} \|^2 = \| (\Pi 1_{S_1})^* \Pi 1_{S_1} \| = \| 1_{S_1} \Pi 1_{S_1} \|$.  The integral kernel of the operator $1_{S_1} \Pi 1_{S_1}$ is the function $\text{vol}( J_H^1)^{-1} 1_{S_1}(x) 1_{J^1_H}(x^{-1} y) 1_{S_1}(y)$, and we apply Proposition \ref{kernelbound} to this.  For any $x$ we have
\[
\text{vol}( J_H^1)^{-1} \int_{K^1_H} 1_{S_1}(x) 1_{J^1_H}(x^{-1} y) 1_{S_1}(y) dy \leqslant \text{vol}( J_H^1)^{-1} \text{vol}( S_1 \cap x J^1_H),
\]
and this is $\ll q^{-l/2} d(\delta, K_{\widetilde{H},w})^{-1/2}$ by Proposition \ref{transverse}.  The corresponding inequality for the $y$-integrals also holds, and so Proposition \ref{kernelbound} gives (\ref{Pi1bound}), and hence (\ref{Aimproved}).  This completes the proof of Proposition \ref{offdiagw}, assuming Proposition \ref{transverse}.

\subsection{Transversality}

We now prove Proposition \ref{transverse}.  We shall only prove the bound for $\textup{vol}(S_1 \cap x J^1_H)$, as the other bound is similar.  If we define $T^1 = T \cap K^1$ and $T_H^1 = T_H \cap K_H^1$, we may assume that $J^1 = g_0 T^1 g_0^{-1} K(\p^l)$ and $J_H^1 = T_H^1 K_H(\p^l)$ by the results of Section \ref{sec:compatpairs}, where $g_0 \in K$ is as in Lemma \ref{adjproj}.

For any $z \in K_H^1$, we define $I_z = \{ t_H \in T_H^1 : zt_H \in S_1 \} \subset T_H^1$.  By disintegrating $x J^1_H$ into cosets of $T^1_H$, we have
\[
\textup{vol}(S_1 \cap x J^1_H) \leqslant \text{vol}( J_H^1) \underset{z \in K_H^1}{\sup} \frac{ \text{vol}(I_z)}{ \text{vol}(T_H^1)},
\]
where the volumes on $T_H$ are taken with respect to the Haar measure obtained as the product of the measures $dt / |t|$ on $F^\times$.  Because this measure satisfies $\text{vol}(T_H^1) = q^{-n}$, it suffices to show that $\text{vol}(I_z) \leqslant 2 d(\delta, K_{\widetilde{H},w})^{-1/2} q^{-l/2-n+1}$ for all $z$.

An elementary manipulation gives that
\[
S_1 = K_H^1 \cap \delta K_H^1 J^1 = K_H^1 \cap \delta K_H^1 g_0 T^1 g_0^{-1} K(\p^l),
\]
and hence
\begin{align}
\notag
I_z & = \{ t_H \in T_H^1 : zt_H \in \delta K_H^1 g_0 T^1 g_0^{-1} K(\p^l) \} \\
\label{Iz}
& = \{ t_H \in T_H^1 : \delta^{-1} zt_H g_0 \in K_H^1 g_0 T^1 K(\p^l) \}.
\end{align}
For simplicity, we denote $\delta^{-1} z$ by $\sigma$.  The properties we shall need of $\sigma$ are that $\sigma \in K^1$, and $d(\sigma, K_{\widetilde{H}}) = d( \delta, K_{\widetilde{H}})$.

Roughly speaking, the condition $\sigma t_H g_0 \in K_H^1 g_0 T^1 K(\p^l)$ in (\ref{Iz}) says that the images of $\sigma t_H g_0$ and $g_0$ in $H \backslash G / T$ are equal `modulo $\p^l$'.  We may therefore study this condition using the invariant functions on $H \backslash G / T$.  For $1 \leqslant i \leqslant n+1$, we define the functions
\[
P_i(g) = g_{(n+1)i}, \quad Q_i(g) = (g^{-1})_{i(n+1)}
\]
on $G$, which are polynomials in $g_{ij}$ and $(\det g)^{-1}$ with integral coefficients.  We note the relation $\sum P_i Q_i = 1$.  These satisfy the equivariance relations $P_i(h g t) = t_i P_i(g)$ and $Q_i(h g t) = t_i^{-1} Q_i(g)$ for $h \in H$ and $t = \text{diag}(t_1, \ldots, t_{n+1}) \in T$, so that $P_i Q_i(g)$ is a function on $H \backslash G / T$.  We have the following lemma.

\begin{lemma}
\label{PQconst}

If $t_H \in I_z$, then we have $P_iQ_i( \sigma t_H g_0) \in P_iQ_i(g_0) + \p^l$ for all $i$.

\end{lemma}

\begin{proof}

By equation (\ref{Iz}), the condition $t_H \in I_z$ implies that $\sigma t_H g_0 \in K_H^1 g_0 T^1 K(\p^l)$, so we may write $\sigma t_H g_0 = k_H g_0 t k$ with $k_H \in K_H^1$, $t \in T$, and $k \in K(\p^l)$.  The entries of $k_H g_0 t$ and $k_H g_0 t k$ lie in $\cO$ and are congruent modulo $\p^l$, and when combined with the fact that the coefficients of $P_i$ and $Q_i$ are integral this gives $P_i Q_i( k_H g_0 t k) \in P_iQ_i( k_H g_0 t) + \p^l$.  The invariance of $P_i Q_i$ implies that $P_iQ_i( k_H g_0 t) = P_i Q_i(g_0)$, which completes the proof.

\end{proof}

We may therefore bound the volume of $I_z$ by finding an $i$ such that the polynomial $P_iQ_i( \sigma t_H g_0)$ on $T_H^1$ is not approximately constant.  If we write $t_H = \text{diag}(t_{H,1}, \ldots, t_{H,n})$, a calculation gives that
\begin{align*}
P_i(\sigma t_H g_0) & = \sum_{j = 1}^n \sigma_{(n+1)j} (g_0)_{ji} t_{H,j} + \sigma_{(n+1)(n+1)} (g_0)_{(n+1)i}, \\
Q_i(\sigma t_H g_0) & = \sum_{j = 1}^n (\sigma^{-1})_{j(n+1)} (g_0^{-1})_{ij} t_{H,j}^{-1} + (\sigma^{-1})_{(n+1)(n+1)} (g_0^{-1})_{i(n+1)}.
\end{align*}
We recall from Lemma \ref{g0entry} that all entries of $g_0$ and $g_0^{-1}$ lie in $\cO^\times$.  When combined with the fact that $\sigma \in K^1$, this implies that the constant terms of $P_i(\sigma t_H g_0)$ and $Q_i(\sigma t_H g_0)$ are in $\cO^\times$, and the norms of the coefficients of $t_{H,j}$ and $t_{H,j}^{-1}$ are equal to $| \sigma_{(n+1)j} |$ and $| (\sigma^{-1})_{j(n+1)} |$.

If we let $d = d(\sigma, K_{\widetilde{H}})$, we have $| \sigma_{(n+1)j} |, \, | (\sigma^{-1})_{j(n+1)} | \leqslant d$ for all $j$, with equality realized in at least one case.  We assume that $| \sigma_{(n+1)j} | = d$ for some $j$, as the case when $| (\sigma^{-1})_{j(n+1)} | = d$ is similar.  It follows that the coefficient of $t_{H,j}$ in $P_i(\sigma t_H g_0)$ has valuation $d$ for all $i$.  If we fix $t_{H,k}$ for $k \neq j$ and consider $P_i(\sigma t_H g_0)$ and $Q_i(\sigma t_H g_0)$ as functions of $t_{H,j}$ alone, we have
\be
\label{tHrational}
P_i Q_i(\sigma t_H g_0) = c_{-1} t_{H,j}^{-1} + c_0 + c_1 t_{H,j},
\ee
where $| c_{-1} | \leqslant d$, $|c_0| = 1$, and $|c_1| = d$.  If we apply Lemma \ref{poly} below to the function (\ref{tHrational}) with $y = P_i Q_i(g_0)$, and combine this with Lemma \ref{PQconst}, we see that the set of $t_{H,j}$ with $(t_{H,1}, \ldots, t_{H,j}, \ldots, t_{H,n}) \in I_z$ has measure at most $2 d^{-1/2} q^{-l/2}$.  Integrating in the other variables gives $\text{vol}(I_z) \leqslant 2 d^{-1/2} q^{-l/2-n+1}$ as required.

\begin{lemma}
\label{poly}

Let $f: 1 + \p \to \cO$ be given by $f(x) = c_{-1} x^{-1} + c_0 + c_1 x$, where $c_i \in \cO$, and $|c_1| = d$.  Then for any $y \in \cO$, we have $\textup{vol} ( f^{-1}( y + \p^l) ) \leqslant 2 d^{-1/2} q^{-l/2}$.

\end{lemma}

\begin{proof}

We may assume that $y = 0$.  We may then replace $f$ with $xf$, and write $x = 1+z$ for $z \in \p$.  This converts the problem to bounding $\text{vol} (g^{-1}(\p^l) )$ where $g(z) = a + bz + cz^2$, and $|c| = d$.  By performing another additive change of variable we may assume that $a = 0$.  If $| g(z) | = |cz(b/c+z)| \leqslant q^{-l}$, then we have either $|z| \leqslant d^{-1/2} q^{-l/2}$ or $|b/c + z| \leqslant d^{-1/2} q^{-l/2}$, and the set of such $z$ has volume $\leqslant 2 d^{-1/2} q^{-l/2}$ as required.

\end{proof}

\section{Proof of the subconvex bound}

We now combine Theorem \ref{mainperiod} with the Ichino--Ikeda formula of Beuzart-Plessis--Chaudouard--Zydor \cite{BPCZ} to deduce Theorem \ref{mainsubconvex}.

\subsection{The unitary Ichino--Ikeda formula}

In this subsection, we let $\pi \times \pi_H$ be a general tempered representation of $G \times H$, and recall the unitary Ichino--Ikeda formula for $\pi \times \pi_H$ as stated in \cite{BPCZ}.  We assume that $\pi \times \pi_H$ has a weak base change to ${\rm GL}_{n+1} \times {\rm GL}_n / E$, denoted $\Pi \times \Pi_H$, that satisfies the conditions of \cite[Section 1.1.3]{BPCZ}.  We refer to Remark \ref{tempered2} for discussion of the existence of such a lift.  We define the product of special $L$-values $\cL( 1/2, \pi \times \pi_H^\vee)$ by
\[
\cL( 1/2, \pi \times \pi_H^\vee) = \prod_{i = 1}^{n+1} \Lambda(i, \eta^i) \frac{ \Lambda(1/2, \Pi \times \Pi_H^\vee) }{ \Lambda(1, \Ad \, \pi) \Lambda(1, \Ad \, \pi_H^\vee) },
\]
where $\eta$ is the quadratic character associated to $E/F$, with corresponding local factors $\cL( 1/2, \pi_v \times \pi_{H,v}^\vee)$.  Our assumption that $\pi$ and $\pi_H$ are tempered implies that $\cL( 1/2, \pi \times \pi_H^\vee)$ and $\cL( 1/2, \pi_v \times \pi_{H,v}^\vee)$ are well-defined and nonzero.

We continue to use the measures on $G(\A)$ and $H(\A)$ defined in Section \ref{sec:measures}, and choose invariant inner products on $\pi_v$ and $\pi_{H,v}$ that factorize the global Peterssen inner product.  Let $\phi = \otimes_v \phi_v \in \pi$ and $\phi_H = \otimes_v \phi_{H,v} \in \pi_H$ be factorizable vectors, and assume that $\phi$ and $\phi_H$, as well as their local factors, have norm one.  We have $\overline{\phi}_H \in \pi_H^\vee$.  Moreover, if we identify $\pi_H^\vee$ and $\pi_{H,v}^\vee$ with the complex conjugate representations $\overline{\pi}_H$ and $\overline{\pi}_{H,v}$ via the Hermitian inner product, then we have local factors $\overline{\phi}_{H,v} \in \overline{\pi}_{H,v} \simeq \pi_{H,v}^\vee$ and a factorization $\overline{\phi}_H = \otimes_v \overline{\phi}_{H,v}$.  For every $v$, we define the normalized local period
\[
\cP_v^\natural(\phi_v \times \overline{\phi}_{H,v} ) = \cL( 1/2, \pi_v \times \pi_{H,v}^\vee)^{-1} \int_{H_v} \langle \pi_v(h) \phi_v, \phi_v \rangle \overline{ \langle \pi_{H,v}(h) \phi_{H,v}, \phi_{H,v} \rangle } dh.
\]
We recall that $\cP_v^\natural(\phi_v \times \overline{\phi}_{H,v} ) = 1$ if all data are unramified, so we may define the product of these periods over all $v$.  The period formula of \cite[Thm 1.1.6.1]{BPCZ} then implies that
\be
\label{IchinoIkeda}
| \cP_H( \phi, \overline{\phi}_H) |^2 = C \cL(1/2, \pi \times \pi_H^\vee) \prod_v \cP_v^\natural(\phi_v \times \overline{\phi}_{H,v} ),
\ee
where $C > 0$ is a constant such that $C$ and $C^{-1}$ are bounded depending on $n$ and the choice of measures.

\subsection{Proof of Theorem \ref{mainsubconvex}}

We recall the notation used in the statement of Theorem \ref{mainsubconvex}, including the finite sets $\Theta_\infty$ and $\Theta_{H,\infty}$ and the family $\cF_\text{SC}$, and we now let $\pi \times \pi_H$ be a representation in $\cF_\text{SC}$.  We will deduce Theorem \ref{mainsubconvex} by applying Theorem \ref{mainperiod} and the period relation (\ref{IchinoIkeda}) to vectors $\phi = \otimes_v \phi_v \in \pi$ and $\phi_H = \otimes_v \phi_{H,v} \in \pi_H$, which will be chosen so that the local periods $\cP_v^\natural(\phi_v \times \overline{\phi}_{H,v} )$ are controlled for all $v$.  As we aim to give an upper bound for  $L(1/2, \Pi \times \Pi_H^\vee)$, we only need to bound these periods from below.

We know that $\cP_v^\natural(\phi_v \times \overline{\phi}_{H,v} ) = 1$ if $v \notin S \cup \{ w \}$.  For $v = w$, we choose $\phi_w$ and $\phi_{H,w}$ to be compatible microlocal lift vectors with norm one.  Lemma \ref{Atrans} then implies that the un-normalized local period
\[
\int_{H_w} \langle \pi_w(h) \phi_w, \phi_w \rangle \overline{ \langle \pi_{H,w}(h) \phi_{H,w}, \phi_{H,w} \rangle } dh
\]
is equal to $\text{vol}( K_H(l) ) \sim q^{-n^2 l}$, which implies that the normalized period also satisfies $\cP_w^\natural(\phi_w \times \overline{\phi}_{H,w} ) \sim q^{-n^2l}$.  For finite primes $v \in S$, we use the local vectors given by following lemma, which is taken from Lemma 5.2 of \cite{Ne1}.

\begin{lemma}

Given $K_v$ and $K_{H,v}$, there is $c > 0$, and subgroups $K_v' < K_v$ and $K_{H,v}' < K_{H,v}$, such that there exist $\phi_v \in \pi_v^{K_v'}$ and $\phi_{H,v} \in \pi_{H,v}^{K_{H,v}'}$ such that $|\cP_v^\natural(\phi_v \times \overline{\phi}_{H,v} )| > c$.

\end{lemma}

Finally, for $v = \infty$, there are a finite number of choices for $\pi_\infty$ and $\pi_{H,\infty}$, and for each we pick $\phi_\infty$ and $\phi_{H,\infty}$ such that $\cP_\infty^\natural(\phi_\infty \times \overline{\phi}_{H,\infty} ) \neq 0$.  With this choice of $\phi$ and $\phi_H$, the period formula (\ref{IchinoIkeda}) gives
\[
q^{n^2 l} | \cP_H( \phi, \overline{\phi}_H) |^2 \gg \cL(1/2, \pi \times \pi_H^\vee).
\]
Moreover, applying Theorem \ref{mainperiod} (with the subgroups $K_v$ and $K_{H,v}$ for finite $v \in S$ replaced with $K_v'$ and $K_{H,v}'$) gives
\be
\label{ellbound}
q^{(n^2 + n-2\delta_0)l} \gg \cL(1/2, \pi \times \pi_H^\vee)
\ee
for any $\delta_0 < \tfrac{1}{4n^2 + 6n + 6}$.

We have
\[
\cL(1/2, \pi \times \pi_H^\vee) \sim \frac{ L(1/2, \Pi \times \Pi_H^\vee) }{ L(1, \Ad \, \pi) L(1, \Ad \, \pi_H^\vee) }
\]
by our assumption that the local factors $\pi_\infty$ and $\pi_{H,\infty}$ lie in fixed finite sets, and also that
\[
L(1, \Ad \, \pi), \quad L(1, \Ad \, \pi_H^\vee) \ll q^{\epsilon l}
\]
which follows from the tempered assumption on $\pi$ and $\pi_H$ by a simple application of the functional equation and Phragm\'{e}n--Lindel\"{o}f.  Combining these with (\ref{ellbound}) gives
\[
q^{(n^2 + n-2\delta_0)l} \gg L(1/2, \Pi \times \Pi_H^\vee),
\]
which is the same as the bound $q^{l n (n+1)(1 - 4 \delta)} \gg L(1/2, \Pi \times \Pi_H^\vee)$ for any $\delta < \tfrac{1}{4n(n+1)(2n^2 + 3n + 3)}$ stated in Theorem \ref{mainsubconvex}.  Finally, as discussed in Remark \ref{conductor}, we have the bound $C(\Pi \times \Pi_H) \ll q^{4l n(n+1)}$, which completes the proof.

\end{document}